\documentclass[12pt]{article}
\usepackage{amsmath,amssymb}
\usepackage{amsthm}
\usepackage{pstricks}
\usepackage{pst-node}
\usepackage{graphicx}
\begin{document}

\title{Group-type Subfactors and Hadamard Matrices}
\author{Richard David Burstein}
\maketitle

\begin{abstract}
A hyperfinite $II_1$ subfactor may be obtained from a symmetric
commuting square via iteration of the basic construction.
For certain commuting squares constructed from
Hadamard matrices, we describe this subfactor
as a group-type inclusion $R^H \subset R \rtimes K$, where $H$ and $K$
are finite groups with outer actions on the hyperfinite
$II_1$ factor $R$.  We find the group of outer automorphisms
generated by $H$ and $K$, and use the method of Bisch and Haagerup
to determine the principal and dual principal graphs.  In some cases
a complete classification is obtained by examining the element
of $H^3(H \ast K / Int R)$ associated with the action.
\end{abstract}

\newcommand{\Z}{\mathbb{Z}}
\newcommand{\B}{\mathcal{B}}
\newcommand{\tr}{ {\rm tr} }

\newcommand{\commsq}[4]{\begin{array}{ccc}#3&\subset&#4\\
\cup& &\cup\\
#1&\subset&#2\end{array}}

\newcommand{\hadsq}[2]{
\commsq{\mathbb{C}}{#1 \mathbb{C}^#2 #1^*}{\mathbb{C}^#2}
{M_{#2}(\mathbb{C})}
}

\newtheorem{lem}{Lemma}[section]

\newtheorem{thm}{Theorem}[section]

\newtheorem{cor}{Corollary}[section]

\theoremstyle{definition}
\newtheorem{dfn}{Definition}

\section{Introduction}

In \cite{JoI}, Jones described the basic construction on a
finite-index subfactor $M_0 \subset M_1$ of type $II_1$.
Iterating this construction gives the tower of factors
$$M_0 \subset M_1 \subset M_2 \subset M_3 \subset ...$$
Taking relative commutants yields two towers of finite
dimensional algebras
$$\begin{array}{ccccccccc}
\mathbb{C}=M_0' \cap M_0&\subset&M_0' \cap M_1&\subset&M_0' \cap M_2&\subset&...\\
                    &       &\cup       &       & \cup      &       & \\
                    &\mathbb{C}=&M_1' \cap M_1&\subset&M_1' \cap M_2&\subset&...\end{array}$$

This is the standard invariant of the subfactor.  The principal
graph and dual principal graph are obtained from the Bratteli
diagrams of these inclusions.  While usually not a complete invariant,
these graphs summarize much important data
about the subfactor.  The standard invariant is a complete invariant
for amenable subfactors~\cite{PoA}.

The classification problem is fundamental in the study of subfactors.
In this paper we address this problem for a family of hyperfinite
$II_1$ subfactors constructed from finite data.  We will provide
principal graphs, and in some cases a full classification up to
subfactor isomorphism.

We recall the definition of commuting squares from~\cite{PoC}.
Let $$\commsq{A}{B}{C}{D}$$ be a quadrilateral of Von Neumann algebras,
with trace.  We may construct the Hilbert space $L^2(D)$ and the
conditional expectations $E_B$, $E_C$ onto $L^2(B)$, $L^2(C)$
respectively.  This quadrilateral is a commuting square if
$E_B$ and $E_C$ commute.

A commuting square is specified
by its four constituent algebras,
the various inclusions, and
certain additional data indicating how the towers
$A \subset B \subset D$ and $A \subset C \subset D$ are related.
This data can be summarized as the biunitary connectionw,
which is an element of the multi-matrix algebra $A' \cap D$.

Two commuting squares
$$\commsq{A}{B}{C}{D}, \quad \commsq{W}{X}{Y}{Z}$$
are isomorphic if there is an algebraic $^*$-isomorphism from
$W' \cap Z$ to $D \cap A'$ which sends $Y \cap W'$ and
$ W' \cap X$ to $A' \cap C$ and $ A' \cap B$ respectively.

Goodman, de la Harpe, and Jones~\cite{GHJ} showed how to
construct a hyperfinite $II_1$ subfactor
from a commuting square of finite-dimensional $C^*$-algebras,
via iteration of the basic construction.
Isomorphic commuting squares produce isomorphic subfactors,
but the converse is not the case.
Certain additional conditions
are imposed on the square: the Bratteli diagrams of
the four inclusions must be connected, and the trace
on the largest algebra should be the unique Markov trace.  In addition
the square should be symmetric, a property which may be determined
from the Bratteli diagrams.  In general, we will
require our commuting squares to have these properties.

As described in~\cite{JS}, the standard invariant of a commuting-square
subfactor is computable to any number of levels in finite time.
However, the time required grows exponentially with the level,
so this method cannot be used to find the full principal graph except
in the most trivial examples.

If the commuting square is flat, then the standard invariant of the
corresponding subfactor may
be found by inspection (see~\cite{EK}).
Likewise,
the standard invariant may be easily computed if the subfactor
is depth 2.  A few
more complex examples have also been studied,
 such as the bipermutation
construction of Krishnan and Sunder~\cite{KS}.  For a general
commuting square, however, even determining finite depth or amenability
of its subfactor is an intractible problem.

A Hadamard matrix $H$ is a real $n$ by $n$ matrix all of whose
entries are $\pm 1$, with $H H^T = n1$.  $n$ must be $1$, $2$,
or a multiple of 4, but it is not known if Hadamard matrices
exist for all such $n$.  These matrices have
been studied for over a century, with connections to areas
as diverse as signals processing, cryptography, and group cohomology.  A complex Hadamard
matrix may be defined similarly
as a unitary matrix all of whose entries have the same complex modulus~\cite{Ho}.

For any complex Hadamard matrix, 
the quadrilateral $$\hadsq{H}{n}$$ commutes, and
induces a commuting-square subfactor.
For $n>6$ many families of such matrices exist, giving a wide
variety of examples of these Hadamard subfactors.
Subfactors of this form were examined in~\cite{JoP}.
Their planar algebras (or equivalently, their standard invariants)
are described by spin models, which
makes computing the first few levels of the standard invariant relatively
straightforward.  For example, the first relative commutant
of a Hadamard subfactor is always abelian~\cite{JoP}.

Only a few Hadamard subfactors have been studied.
The $n \times n$ Fourier matrix is defined by $F_{ij}=\xi^{ij}$,
where $\xi$ is a primitive $n$th root of unity.
It may be easily computed using the profile matrix of~\cite{JoP}
that Fourier matrices
and their tensor products give depth-2 subfactors.
Some other examples of small index are studied by
Camp and Nicoara in~\cite{CN},
with full computations of the principal graph
for a few index-4 examples.
For practically all Hadamard subfactors, nothing is known about
the standard invariant beyond the first few levels.

In this paper we will partially
classify a family of Hadamard
subfactors obtained from tensor products of Fourier matrices
(and slight generalizations thereof) with a certain additional twist.
This twisted tensor product construction was suggested to the author
by Jones.
As we will show, these
subfactors may be described as inclusions
of the form $R^H \subset R \rtimes K$, for appropriate actions
of finite abelian groups $H$ and $K$ on the hyperfinite $II_1$
factor $R$.

These group type subfactors were studied by Bisch and Haagerup
in~\cite{BH}.  In this paper,
the authors give a method for computing the principal graph of any
such subfactor from the image of the free product of $H$ and $K$
in $Out R$.

To analyze our Hadamard examples, we will first (section 2) discuss
automorphisms of the hyperfinite $II_1$ factor which are compatible
with the structure of the Jones tower in a particular way.  For
such automorphisms, determining outerness reduces to a problem
in finite-dimensional linear algebra.

In section 3 we will discuss those Hadamard matrices which produce
depth-2 subfactors.  Taking the twisted tensor product of two
such matrices gives a new Hadamard matrix, whose subfactor
is of Bisch-Haagerup type.  We will show that in this case
the group actions have the compatibility property mentioned above.

The principal graph is not a complete invariant even for finite-depth
subfactors.  To classify the 
Bisch-Haagerup subfactors
up to isomorphism, it is also necessary to consider a certain
scalar 3-cocycle $\omega \in H^3(H \ast K / Int R)$
associated with the group action.
We will discuss this cohomological data in section 4, using
on the conjugacy invariants of ~\cite{JoT}.

In section 5 we will use these methods to describe several examples.
As well as examining certain finite-depth cases,
we will provide infinite-depth Hadamard subfactors of every composite
index.  We will use the results of section 4 to fully classify
all index-4 Hadamard subfactors.

Throughout this paper, if $A$ is a Von Neumann algebra acting
on a Hilbert space $H$, we will take $A'$ to be the commutant of $A$ in
the set $\B(H)$ of bounded linear operators on $H$.
$A' \cap A = Z(A)$ is the center of $A$.

\section{Compatible Automorphisms of the Hyperfinite $II_1$ Factor}

\subsection{Introduction}

Let $B_0 \subset B_1$ be a connected inclusion
of finite-dimensional $C^*$ algebras, along with its
Markov trace.  Iterating the basic
construction gives a tower of algebras
$$B_0 \subset B_1 \subset B_2 \subset ... \subset B_{\infty}$$
where $B_{\infty}$ is the hyperfinite $II_1$ factor.
We examine
automorphisms of $B_{\infty}$ which are compatible
with the structure of the tower.

\subsection{Properties of the Jones Tower}

We recall some basic properties of the iterated basic construction
on finite-dimensional Von Neumann algebras.  This discussion
is largely taken from~\cite{JS}.

Let $B_0 \subset B_1$ be a connected inclusion of finite-dimensional Von
Neumann algebras.  An inclusion is connected if the commutant
$B_0 \cap B_1'$ is equal to $\mathbb{C}$.  Defining a trace $\tr$ on $B_1$
makes $L^2(B_1)$ into a Hilbert space with inner product $<x,y>=\tr(y^*x)$,
on which $B_1$ acts by left multiplication.
We may then define the conditional expectation $e=E_{B_0}$, which
is the orthogonal projection onto the closed subspace
$L^2(B_0) \subset L^2(B_1)$.
This allows us to perform the basic construction on
$B_0 \subset B_1$, obtaining
$B_2 = \{B_1, e \}'' \subset \B(L^2(B_1))$.
If we extend $\tr$ to a trace on $B_2$,
we may then iterate
this procedure, obtaining the Jones tower
$$B_0 \subset B_1 \subset B_2 \subset B_3 \subset ...$$
There is a unique trace on the original $B_1$ (the Markov trace, of
some modulus $\tau$)
which extends to a trace on the entire tower.  With this choice of trace,
we may apply the GNS construction and take the closure
of $\cup_i B_i$ to obtain the hyperfinite $II_1$ factor
$B_{\infty}$.  We label the Jones projections
by $B_i=\{B_{i-1}, e_i\}''$ for $i \ge 2$.
Then $\tr(e_i)=\tau$ for all $i$, and the $e_i$'s
obey the relations

$e_ie_j=e_je_i$ for $|i-j|>1$

$e_ie_{i \pm 1}e_i=\tau e_i$

We also have $e_i x e_i = E_{B_{i-2}}(x)$ for $x \in B_{i-1}$.

\subsection{Definitions and Basic Properties}

Throughout this section we will require all of our automorphisms
to be trace-preserving and respect the adjoint operation.

\begin{thm}
Let the tower of $B_i$'s be as above.
Let $a$ be an automorphism of
$B_1$ which leaves $B_0$ invariant.
Then there is a unique (trace-preserving, $^*$-)
automorphism $\alpha$ of $B_{\infty}$
such that $\alpha(e_i)=e_i$, $\alpha(B_i)=B_i$, and $\alpha|_{B_1}=a$.
\end{thm}
\begin{proof}
We construct $\alpha$ inductively.
Let $\alpha_1=a$.  Let $\alpha_i$ be a trace-preserving
$^*$-automorphism of
$B_i$, leaving $B_j$ invariant and fixing $e_j$ for $i \le j$.  For
$x, y \in B_i$, 
define $\alpha_{i+1}$ by
$\alpha_{i+1}(xe_{i+1}y)=\alpha_i(x)e_{i+1}\alpha_i(y)$.
It is not immediately clear that this is a defined map; in
principal
we might have $\sum_k x_k e_{i+1} y_k = 0$ with
$\sum_k \alpha_i(x_k) e_{i+1}\alpha(y_k) \neq 0$, but we will
show that this possibility does not arise.

We compute $\alpha_{i+1}( a e_{i+1} b)\alpha_{i+1}(c e_{i+1} d)$.
This is
$$\alpha_i(a)e_{i+1}\alpha_i(b)\alpha_i(c)e_{i+1} \alpha_i(d)$$
$$=\alpha_i(a)E_{B_{i-1}}(\alpha_i(bc))e_{i+1}\alpha_i(d)$$
$\alpha_i$ leaves $B_{i-1}$ invariant, and therefore commutes with
$E_{B_{i-1}}$, so this is
$$\alpha_i(a)\alpha_i(E_{B_{i-1}}(bc))e_{i+1}\alpha_i(d)$$
$$=\alpha_{i+1}(aE_{B_{i-1}}(bc)e_{i+1}d)$$
$$=\alpha_{i+1}(a e_{i+1} b c e_{i+1} d)$$
Therefore $\alpha_{i+1}$ is a homomorphism on elements of the
form
$a e_{i+1} b$; $B_i e_{i+1} B_i = B_{i+1}$~\cite{JS}, so
$\alpha$ is a homomorphism on all of $B_{i+1}$.

$\alpha_{i+1}$
clearly sends adjoints to adjoints.  Since $\alpha_i$ is a
trace-preserving homomorphism,
we have $$\tr(xe_{i+1}y)=\tr(e_{i+1})\tr(yx)=\tr(e_{i+1})\tr(\alpha_i(yx))=
\tr(\alpha_i(x)e_{i+1}\alpha_i(y))$$
by the properties of the Jones projections.  So $\alpha_{i+1}$
is trace-preserving as well.  A trace-preserving $^*$-homomorphism
is an isometry.
This tells us that $\alpha_{i+1}$ is a defined map
on $B_{i+1}$,
since $||x||_2 = 0 \rightarrow ||\alpha_{i+1}(x)||_2=0$,
and that it is injective, since
$||\alpha_{i+1}(x)||_2=0 \rightarrow ||x||_2=0$.

Therefore
$\alpha_{i+1}$ is a (trace-preserving,
$^*$-) automorphism of $B_{i+1}$,
which leaves $B_j$ invariant and fixes $e_j$ for $j \le i+1$.

For $x \in B_i$, 
$\alpha_{i+1}(x)e_{i+1}=\alpha_{i+1}(xe_{i+1})=\alpha_i(x)e_{i+1}$.
Since $\alpha_i(x)$ and $\alpha_{i+1}(x)$ are in $B_i$, by properties of
the Jones projections this means
$\alpha_{i+1}(x)=\alpha_i(x)$.  In other words $\alpha_{i+1}|B_i=\alpha_i$.
So $\alpha_{i+1}$ is an extension of $\alpha_i$, and
we can define $\alpha_{\infty}$ on $\cup_i B_i$ by
$\alpha_{\infty}(x)=\alpha_j(x)$ for $x \in B_j$.
All the $\alpha_i$'s are norm-1, so $\alpha_{\infty}$ is bounded in
norm.  This means it
extends to the closure of $\cup_i B_i$, giving an automorphism $\alpha$
of $B_{\infty}$.

$\alpha|_{B_i}=\alpha_i$, so $\alpha|_{B_1}=\alpha_1=a$,
and $\alpha(e_i)=\alpha_i(e_i)=e_i$.

Uniqueness of $\alpha$
is immediate since $B_1$ and the $e_i$'s generate $B_{\infty}$.
\end{proof}

Alternatively, if $\alpha$ is a automorphism of $B_{\infty}$
which leaves the $B_i$'s invariant and fixes the $e_i$'s, then it is equal to
the extension of $\alpha|_{B_1}$ to $B_{\infty}$ as above.

These conditions are a bit stronger than necessary.  If
$\alpha$ fixes the $e_i$'s and leaves $B_1$ invariant,
then it also leaves $\{e_2\}' \cap B_1  = B_0$ and
$\{B_1, e_2, ... , e_i\}''=B_i$ invariant.

Maps of this form may be said to be compatible with the tower.

\begin{dfn}
If $\alpha \in Aut(B_{\infty})$ fixes $e_i$ and leaves
$B_i$ invariant for all $i$, then $\alpha$ is a
\begin{bf}compatible\end{bf}
automorphism.
\end{dfn}

In~\cite{Loi}, the author examined automorphisms of a $II_1$
factor $M_1$ which send a subfactor $M_0 \subset M_1$ to itself.
Such an automorphism $\alpha$
extends to the Jones tower of $M_0 \subset M_1$ by taking $\alpha(e_i)=e_i$,
and restricts to the tower of relative commutants $\{M_0' \cap M_k\}$.
The above definition of compatible automorphisms may be thought
of as a finite-dimensional version of Loi's construction.

\subsection{The Canonical Shift}

The Bratteli diagram of an inclusion of finite-dimensional Von
Neumann algebras is a graphical depiction of the inclusion matrix
(see~\cite{JS}).
For $0 \le i \le j$, both $B_i' \cap B_j$ and $B_{i+2}' \cap B_{j+2}$
may be implemented as the algebra of length $(j-i)$ paths
on the Bratteli diagram of
$B_0 \subset B_1$~\cite{JS}.  It follows that these two algebras are isomorphic.
The isomorphism is the canonical shift, and we may construct it
as follows.

\begin{lem}

Let
$$T_{ij} = \tau^{-(j-i)/2}e_{j+2}e_{j+1}...e_{i+3}e_{i+2}$$
For all $x \in B_i' \cap B_j$, there is a unique $y \in B_{i+2}'
\cap B_{j+2}$ such that $T_{ij} y T_{ij}^* = e_{j+2} x $,
and the map $\Theta$ defined by $y = \Theta(x)$ is a
$^*$-isomorphism from $B_i' \cap B_j$ to $B_{i+2}' \cap B_{j+2}$.
\end{lem}
\begin{proof}

We have

$$T_{ij}^*T_{ij}=\tau^{-(j-i)}e_{i+2}e_{i+3}...e_{j+1}e_{j+2}e_{j+1}...e_{i+2}$$
which
is equal to $e_{i+2}$ by the properties of the Jones projections.

Then let $x$ be an element of $B_{i+2}' \cap B_{j+2}$.
$T_{ij}xT_{ij}^*$ is an element of $e_{j+2}B_{j+2}e_{j+2}$.  From the
properties of the basic construction, it follows that there is
a unique $y$ in $B_k$ such that
$T_{ij}xT_{ij}^*=e_{j+2}y$.  Since $T_k$, $x$, and $e_{j+2}$
all commute with $B_i$, $y$ must do so as well.

Now we may define $\rho_{ij}:B_{i+2}' \cap B_{j+2} \rightarrow
B_i' \cap B_j$ by
$T_{ij}xT_{ij}^*=e_{j+2} \rho_{ij}(x)$.  It follows immediately
that $\rho(x^*)=\rho(x)^*$.

For $x, y \in B_{i+2}' \cap B_{j+2}$,
we then have
$$T_{ij}xT_{ij}^*T_{ij}yT_{ij}^*=e_{j+2} \rho_{ij}(x) e_{j+2}\rho_{ij}(y)$$
Since $\rho_{ij}(x)$ is in $B_j$, it commutes with $e_{g+2}$, and this is
$e_{g+2}\rho_{ij}(x)\rho_j(y)$.

We also have $T_{ij} x T_{ij}^* T_{ij} y T_{ij}^* = T_{ij}xe_{i+2}yT_{ij}^*$.
Since $x$ commutes with $e_{i+2}$ and $T_{ij}e_i=T_{ij}$,
this implies that $T_{ij}xyT_{ij}^*=e_{j+2}\rho_{ij}(x)\rho_{ij}(y)$, so
$e_{j+2}\rho_{ij}(xy)=e_{j+2}\rho_k(x)\rho_k(y)$.  $e_{j+2}a=0$
for $a \in B_j$ only if $a=0$, so this means that $\rho_{ij}$ is a homomorphism.

Now we investigate the norm of $\rho_{ij}$.
Let $x$ be an element of $B_{i+2}' \cap B_{j+2}$.  We have
$\tr(e_{j+2}\rho_{ij}(x))=\tau \tr(\rho_{ij}(x))$.  This is the
same as $\tr(T_{ij}xT_{ij}^*)=\tr(xT_{ij}^*T_{ij})=\tr(x e_{i+2})$.
There is
a trace-preserving conditional expectation onto $B_{i+2}' \cap B_{j+2}$,
so we find
$$\tr(x e_{i+2}) = \tr(E_{B_{i+2}' \cap B_{j+2}}(x e_{i+2}))=
\tr(x E_{B_{i+2}' \cap B_{j+2}}(e_{i+2}))$$

The quadrilateral $$\commsq{Z(B_{i+2}) }{B_{i+2}}
{B_{i+2}' \cap B_{j+2}}{B_{j+2}}$$
commutes, since the conditional expectation onto $B_{i+2}$
preserves $B_{i+2}'$.  So $E_{B_{i+2}' \cap B_{j+2}}(e_{i+2})$
$=E_{Z(B_{i+2})}(e_{i+2})$.

\begin{dfn}
The \begin{bf}central support\end{bf} of a projection $p$
in a finite Von Neumann algebra $A$ is the smallest central
projection $q \in Z(A)$ with $pq=p$.  If $q=1$, we say
that $p$ has full central support.
\end{dfn}

$e_{i+2}<1$, since it is a projection.
Also $e_{i+2}$ has full central support
in $B_2$~\cite{JS}.  Since $Z(B_2)$ is finite-dimensional, this means
there is some $\lambda>0$ with
$\lambda 1< E_{Z(B_{i+2})}(e_{i+2})<1$.  Therefore
$\lambda \tr(x)<\tr(x E_{Z(B_{i+2})}(e_{i+2}))<\tr(x)$,
implying $\tau^{-1}\lambda \tr(x)<\tr(\rho_{ij}(x))<\tau^{-1}\tr(x)$.
We then find $c>0$ with $c<\tau^{-1}\lambda,c^{-1}>\tau^{-1}$.

$\rho_{ij}$ is a $^*$-homomorphism, so for any $x \in B_{i+2}' \cap
B_{j+2}$
we have
$$||\rho_{ij}(x)||_2^2 = \tr(\rho_{ij}(x)^*\rho(x))=\tr(\rho(x^*x))$$
The above inequality on trace then implies
$$c||x||_2^2<||\rho_{ij}(x)||_2^2<c^{-1}||x||_2^2$$

$\rho_{ij}$ is thus an injective homomorphism.
$B_{i+2}' \cap B_{j+2}$ is isomorphic to $B_i' \cap B_j$, so
in fact $\rho_{ij}$ is an isomorphism.
It follows that
there exists $\Theta_{ij}:B_i' \cap B_j \rightarrow B_{i+2}' \cap B_{j+2}$
such that $\rho_{ij}\Theta_{ij}$ is the identity, and $\Theta_{ij}$
is also a $^*$-isomorphism.

By the definition of $\rho_{ij}$,
$\Theta_{ij}(x)$ is then the unique element of $B_{i+2}' \cap B_{j+2}$
obeying the relation
$T_{ij} \Theta_{ij}(x) T_{ij}^* = e_{j+2}x$ for $x \in B_i' \cap B_j$.
\end{proof}

Using the same constant $c$ as above, we must have
 $$c||x||^2_2<||\Theta_{ij}(x)||^2_2<c^{-1}||x||^2_2$$ 
 since $\Theta_{ij} \rho_{ij}$ is the identity.
This map $\Theta_{ij}$ is the canonical shift on $B_i' \cap B_j$.

\subsection{The Iterated Shift}

We now recall some results from~\cite{JS} and ~\cite{GHJ}, based on
Perron-Frobenius theory.

Let $\mathbf{s^{(i)}}$
be the size vector for $B_i$, i.e. the $x$th minimal
central projection $p_x \in B_i$ has $p_x B_i = M_{s^{(i)}_x}(\mathbb{C})$.
Then as $n$ goes to infinity,
$\tau^{2n} \mathbf{s^{(i+2n)}}$ converges to some vector $\mathbf{v}$,
which is a Perron-Frobenius eigenvector for
the inclusion matrix of $B_i \subset B_{i+2}$.
Every component of each $\mathbf{s^{(k)}}$ is positive, and this
is true of $\mathbf{v}$ as well,
 so for all $x$ labeling a central projection of $B_i$,
 the set
$\{\tau^{n}(s^{(i+2n)})_x|n \in \mathbb{N})\}$ is bounded and bounded
away from zero.  Since $Z(B_i)$ is finite-dimensional, there is
some constant $c>0$ with
$c < \tau^{n}(s^{(i+2n)})_x < c^{-1}$ for all $n \in \mathbb{N}$,
$1 \le x \le {\rm dim}Z(B_i)$.

Likewise, let $\mathbf{t^{(j)}}$ be the trace vector for $B_j$,
with $t^{(j)}_y$ equal to the trace of a minimal projection
in $p_y B_j$.  From the Markov property of the trace
on $B_j \subset B_{j+1}$, we have $\mathbf{t^{(j+2n)}} =
\tau^{n} \mathbf{t^{(j)}} $~\cite{GHJ}.
Again, finite dimensionality of $Z(B_j)$ implies
that there is $d>0$ with
$d < \tau^{-n} t^{(j)}_y < d^{-1}$ for all $n \in \mathbb{N}$,
$1 \le y \le {\rm dim} Z(B_j)$.

This implies that the traces of certain projections in the
tower of relative commutants are bounded away from zero.

\begin{lem}
Choose $0 \le i \le j$.
There exists $\epsilon>0$ such that for all $n \ge 0$ and $p>0$ a projection
in $B_{i+2n}' \cap B_{j+2n}$, $\tr p \ge \epsilon$.
\end{lem}
\begin{proof}

Let $p$ be a minimal projection in $B_{i+2n}' \cap B_{j+2n}$.

From the path algebra model of \cite{JS}, the trace of $p$ is equal
to $s^{(i+2n)}_x t^{(j+2n)}_y$ for some $1 \le x \le
{\rm dim} Z(B_i)$, $1 \le y \le {\rm dim} Z(B_j)$.

The above Perron-Froebenius argument implies that there are constants
$c>0$, $d>0$ such that $c < \tau^{-n}
s^{(i+2n)}_x$, $d < \tau^n t^{(i+2n)}_y$ for all
$x$, $y$, $n$.
This means that $\epsilon = cd < \tr p$.
\end{proof}

The {\bf iterated shift} $\Theta_{ij}^n$ is defined as
$$\Theta_{ij}^n =
\Theta_{i+2(n-1),j+2(n-1)} \Theta_{i+2(n-2),j+2(n-2)}...\Theta_{i+2,j+2}
\Theta_{ij}$$
This is a $^*$-isomorphism from $B_i' \cap B_j$ to
$B_{i+2n}' \cap B_{j+_2n}$.

\begin{thm}
For all $i,j$ there exists $c>0$ such that for all $x \in
B_i' \cap B_j$ and all $n>0$ we have
$$c ||x||_2 \le ||\Theta^n_{ij}(x)||_2 \le c^{-1}||x||$$
\end{thm}
\begin{proof}
Let $p$ be a minimal projection in $B_i' \cap B_j$.
By the previous lemma we have $\epsilon>0$ such that
$\epsilon \le \tr p \le 1$, $\epsilon \le \tr \Theta_{ij}^n(p) \le 1$
for all $n$.
It follows that
$$\epsilon^2 \tr p \le \tr \Theta_{ij}^n(p) \le \epsilon^{-2} \tr p$$

Any positive element is a linear combination of minimal projections,
so this inequality holds for all $a>0$ in $B_i' \cap B_j$.  Applying
this to $x^*x$ we get
$$\epsilon^2 \tr x^*x \le \tr\Theta_{ij}^n(x)^*\Theta_{ij}^n(x) \le
\epsilon^{-2} \tr x^*x$$
since $\Theta_{ij}^n$ is a $^*$-isomorphism.  This gives
$$\epsilon||x||_2 \le ||\Theta_{ij}^n(x)||_2 \le \epsilon^{-1}||x||_2$$
\end{proof}

\subsection{The Iterated Shift and Central Sequences}

Let $\omega$ be a free ultrafilter of the natural numbers.
If $R$ is the hyperfinite $II_1$ factor, we define the ultrapower
$R^{\omega}$ as the set of
bounded functions from the natural numbers to $R$, modulo
those which approach zero strongly along the ultrafilter.
Convergence along the ultrafilter is defined using
the ultralimit(see ~\cite{EK}): for a sequence of points
$(x_i)$ in some topological space,
we say that $\lim_{i \rightarrow \omega} (x_i) = L$
if for any neighborhood $N$ of $L$
there is a set $S \subset \mathbb{N}$ in the ultrafilter such that
$x_i \in N$ for all $i \in S$.  

$R$ embeds in $R^{\omega}$ as constant sequences.  The
central sequence algebra $R_{\omega}$ is then defined as the subalgebra
$R' \cap R^{\omega}$,
and both $R^{\omega}$ and $R_{\omega}$ are nonseparable $II_1$
factors~\cite{EK}.
If $x = (x_i)$ is an
element of $R^{\omega}$, then $\tr(x)$ is
$\lim_{i \rightarrow \omega} \tr(x_i)$.

Take $0 \le i \le j$.
Theorem 2.2
gives a map from $B_i' \cap B_j$ into
the central sequence algebra $(B_{\infty})_{\omega}$.

\begin{lem}
Let $\tilde{\Theta}$ from $B_i' \cap B_j$ to
$l^{\infty}(B_{\infty})$ defined by
$\tilde{\Theta}(x) = ( \Theta_{ij}^n(x) )$.
Then $\tilde{\Theta}$ is an injective homomorphism
from $B_i' \cap B_j$ into $(B_{\infty})_{\omega}$.
\end{lem}
\begin{proof}
From theorem 2.2, the sequence
$\tilde{\Theta}(x)$ is bounded in $\infty$-norm,
so it defines
an element of $B_{\infty}^{\omega}$.

From the definition of the iterated shift $\Theta_{ij}^n$,
this element 
asymptotically commutes with all the $B_i$'s, i.e.
$\lim_{n \rightarrow \omega}||[\Theta_{ij}^n(x), y]||_2=0$
for $y$ in any $B_i$.
Since the union of the $B_i$'s are dense in $B_{\infty}$,
$\tilde{\Theta}(x)$ asymptotically commutes with every element
of $B_{\infty}$ and is contained in $(B_{\infty})_{\omega}$.

$\tilde{\Theta}$ is a homomorphism,
since each $\Theta_{ij}^n$ is.
From lemma 2.2,
the iterated shift is bounded away from
zero in 2-norm.  So if $x \neq 0$, the sequence $\tilde{\Theta}(x)$
does not approach zero in 2-norm, and gives a nonzero element of
the central sequence algebra.  In other words, $\tilde{\Theta}$ is
injective.
\end{proof}

Since $x$ has finite spectrum, this implies that
$\tilde{\Theta}$ preserves the spectrum of $x$.
In particular $||x||_{\infty} = \tilde{\Theta}(||x||)_{\infty}$.

\subsection{Outerness of Compatible Actions}

Let $\alpha$ be a compatible automorphism.
Take $0<i<j$, $x \in B_{i+2}' \cap B_{j+2}$.
Then since $\alpha$ fixes the Jones projections,
we have
$$e_{j+2}\alpha(\rho_{ij}(x)) = \alpha(e_{j+2} \rho_{ij}(x))
= \alpha(T_{ij} x T_{ij}^*) = \alpha(T_{ij}) \alpha(x) \alpha (T_{ij}^*) =
T_{ij} \alpha(x) T_{ij}^*$$
This is the same as $e_{j+2} \rho_{ij}(\alpha(x))$.
So $\alpha$ commutes with $\rho_{ij}$ for all $i, j$.
It follows that $\alpha$ commutes with each $\Theta_{ij}$ as well.
Since inner automorphisms act trivially on central sequences,
this gives us a test for outerness of compatible automorphisms.

\begin{lem}
Let the tower of $B_i$'s be as above.
If $\alpha$ is a compatible automorphism of $B_{\infty}$,
and $\alpha$ does not act trivially on $B_0' \cap B_i$ for all $i$,
then $\alpha$ is outer.
\end{lem}
\begin{proof}
Let $x$ be an element of $B_0' \cap B_i$, for some $i \ge 0$.
Suppose
that $\alpha(x) \neq x$.  Then
$\alpha(x)-x$ is a nonzero element of $B_0' \cap B_i$,
and so by lemma 2.3 $\tilde{\Theta}(\alpha(x)-x)$
is a nonzero element
of the central sequence algebra $(B_{\infty})_{\omega}$.

$\alpha$ has a pointwise action on $(B_{\infty})^{\omega}$ which
restricts to $(B_{\infty})_{\omega}$.
Since $\alpha$ commutes with $\Theta$ and $\Theta^n$,
we have
$$\alpha(\tilde{\Theta}(x)) = (\alpha(\Theta_{0i}^n(x))) =
(\Theta_{0i}^n(\alpha(x))) = \tilde{\Theta}(\alpha(x))$$
$\tilde{\Theta}$ is injective from the previous section, so
we have $\alpha(\tilde{\Theta}(x))-\tilde{\Theta}(x)) \neq 0$
as well.  This means that the induced
action of $\alpha$ on central sequences is nontrivial.  Inner automorphisms
act trivially on central sequences, so with the above assumption
$\alpha$ is outer.
\end{proof}

We may conclude that if $\alpha$ is not outer, i.e., $\alpha = Ad u$
for some unitary $u \in B_{\infty})$,
it must fix $B_0' \cap B_i$ for all $i$.  This means
that $u$ commutes with $B_0' \cap B_i$ for all $i$,
and hence with the strong closure
$\overline{\cup_i^{\infty} B_0' \cap B_i}^{st}$.

\begin{lem}
Let the tower of $B_i$'s be as above.  Then
$\overline{\cup_i^{\infty} B_0' \cap B_i}^{st}=
B_0' \cap B_{\infty} $.
\end{lem}
\begin{proof}
The following square commutes:
$$\commsq{B_0' \cap B_i}{B_0' \cap B_{\infty}}{B_i}{B_{\infty}}$$

The $B_i$'s are dense in $B_{\infty}$, so
$||x - E_{B_i}(x)||_2$ goes to zero as $i$ goes to infinity.
$E_{B_i}(x) = E_{B_0' \cap B_i}(x)$, so
$x$ is in the $2$-norm closure of $\cup_i B_0' \cap B_i$.
A sequence
of elements in $B_{\infty}$ converges strongly if it converges
in 2-norm,
implying that $x$
is in the strong closure of $\cup_i B_0' \cap B_i$.
\end{proof}

These results imply that if $Ad u$ is compatible
inner, then $u$ must commute
with $B_0' \cap B_{\infty}$.  Finite-dimensional algebras in a
$II_1$ factor have the bicommutant property, so $u$ must be in $B_0$
if $Ad u$ is compatible.
Since compatible automorphisms are determined by their restriction to $B_1$,
we may make a slightly stronger statement, as follows:

\begin{thm}
Let the tower of $B_i$'s be as above.
If $\alpha$ is a compatible automorphism of $B_{\infty}$, then $\alpha$
is inner if and only if $\alpha|_{B_1} = Ad u|_{B_1}$ for some unitary
$u \in B_0$.
\end{thm}
\begin{proof}
First suppose that $\alpha$ is compatible, and
$\alpha|_{B_1}=Ad u|_{B_1}$ for some unitary $u \in B_0$.  Then
$\alpha$ agrees with $Ad u$ on $B_1$, and both
automorphisms fix the $e_i$'s.  $B_1$ and the $e_i$'s generate
$B_{\infty}$, so in this case $\alpha = Ad u$ and is inner.

Alternatively, let $\alpha$ be inner and
compatible.  Then $\alpha = Ad u$ for some unitary $u \in B_0$,
and $\alpha|_{B_1}=Ad u|_{B_1}$.
\end{proof}

This theorem reduces determining outerness of a compatible
automorphism to a purely computational problem. 

\section{Commuting-square subfactors and group actions}

\subsection{Introduction}

In~\cite{BH}, the authors discuss 
group type subfactors of the form $M^H \subset M \rtimes K$, where
$H$ and $K$ are finite groups with outer actions on $M$.  The principal and dual principal
graphs of such subfactors may be computed by finding the quotient
$G = H \ast K / Int M$.  This requires being able to determine
whether a specified word $w \in H \ast K$ produces an outer automorphism.
In general this may be difficult, even if $M$ is hyperfinite.

We will apply this technique to the commuting-square subfactors mentioned
in the introduction.
We will give conditions for a commuting
square subfactor to be of fixed-point or crossed-product type,
and describe how to compose two such subfactors to obtain a Bisch-Haagerup
subfactor.  As we will see, in this case the action of $H$ and $K$
is compatible with the Jones tower of the intermediate
subfactor.  This will allow us to use the results of the previous
section to classify many previously intractable examples.

We briefly review the basic construction on commuting squares,
following~\cite{JS}.
A quadrilateral of Von Neumann algebras (with trace)
$$\commsq{A_0}{A_1}{B_0}{B_1}$$
is a commuting square if $E_{B_0}$ and $E_{A_1}$ commute as
operators on $L^2(B_1)$.  If we use the Markov
trace on $B_0 \subset B_1$, we may iterate
the basic construction on $B_0 \subset B_1$ as in the previous section
to obtain a hyperfinite $II_1$ factor.  We obtain a tower of $A_i$'s
as well, given by $A_i = \{A_{i-1}, e_i\}''$.

$e_i$ implements the conditional expectation from
$A_{i-1}$ to $A_{i-2}$. By~\cite{JS},
for the inclusion $A_i \subset A_{i+1} \subset A_{i+2}$ to be isomorphic
to the basic construction on $A_i \subset A_{i+1}$
it is then sufficient for the ideal $A_{i-1}e_iA_{i-1}$ to include the identity.
From the properties of the Jones projections, this is true if the
ideal $A_1 e_2 A_1$ includes the identity as operators on
$L^2(B_1)$.  This is the symmetry property of~\cite{JS}; the
authors give several descriptions of this property, which they
show are all equivalent to the following.

\begin{dfn}
A commuting square $$\commsq{A_0}{A_1}{B_0}{B_1}$$
is \begin{bf}symmetric\end{bf} if $1 \in A_1 E_{B_0} A_1$ as
operators on $L^2(B_1)$.
\end{dfn}

The Markov
trace on $B_0 \subset B_1$ extends to $B_{\infty}$, and
then restricts to $A_{\infty} \subset B_{\infty}$ with
no additional assumptions, producing a hyperfinite
$II_1$ subfactor.  The index $[B_{\infty}:A_{\infty}]$ of this
inclusion is the squared norm of the inclusion matrix for the
algebras
$A_0 \subset B_0$.  Every symmetric
connected commuting square admits a unique
Markov trace, so from now on we will assume that
this is the trace we use for any such commuting square.

In order for $B_{\infty}$ and $A_{\infty}$ to be factors,
the horizontal inclusions in the above square must
be connected.  However, we will not require the vertical
inclusions to be connected, since we are not concerned
here with the vertical basic construction.

\subsection{Groups of Outer Automorphisms For Composite Commuting Squares}

\begin{thm}
Consider the symmetric, horizontally connected commuting square
$$\commsq{A_{00}}{A_{10}}{A_{01}}{A_{11}}$$ generating a subfactor
$N \subset M$ via horizontal iteration of the basic construction, with
Jones projections $\{e_i\}$.
Suppose that there exist intermediate algebras $B_0$, $B_1$, as follows:

$$\begin{array}{ccc}
A_{01}&\subset&A_{11}\\
\cup& &\cup\\
B_0&\subset&B_1\\
\cup& &\cup\\
A_{00}&\subset&A_{10}\end{array}$$

Assume $B_0 \subset B_1$ is connected, and the quadrilateral
$$\commsq{B_0}{B_1}{A_{01}}{A_{11}}$$ commutes.
Then there is an intermediate subfactor $P$ obtained by
iterating the basic construction on the $B_i$'s, and both
$N \subset P$ and $P \subset M$ arise from sub-commuting-squares
of the original diagram.
\end{thm}
\begin{proof}
The ideal $B_1 E_{A_{10}} B_1$ necessarily contains the identity, since
the original commuting square is symmetric and $A_{10} \subset B_1$.
Therefore the upper commuting square $$\commsq{B_0}{B_1}{A_{01}}{A_{11}}$$ is
symmetric, and is Markov by hypothesis.
Let $B_{i+1}=\{B_i, e_{i+1}\}''$; then by the symmetric
property all inclusions
$B_i \subset B_{i+1} \subset B_{i+2}$ are standard.  The Markov
trace on the $A_{i1}$'s restricts to one on the $B_i$'s~\cite{JS}, and we obtain
an intermediate subfactor $N \subset P = \overline{\cup_i B_i}^{st} \subset M$.

The lower quadrilateral $$\commsq{A_{00}}{A_{10}}{B_0}{B_1}$$ automatically
commutes, since $E_{A_{10}}(B_0) \subset E_{A_{10}}(A_{01})=A_{00}$.
To compute $A_{10}E_{B_0}A_{10}$ as operators on $L^2(B_1)$, we note
that $1 \in A_{10}E_{A_{01}}A_{10}$ as operators on $L^2(A_{11})$.
Multiplying both sides by $E_{B_1}$, we find that
$E_{B_1} \in A_{10}E_{B_1}E_{A_{01}}A_{10}$, since
$E_{B_1}$ commutes with $A_{10}$.   Since
the upper quadrilateral commutes by assumption,
we have $E_{B_1}E_{A_{01}}=E_{B_0}$.  Therefore
$E_{B_1} \in A_{10}E_{B_0}A_{10}$, as operators on $L^2(A_{11})$.
If we restrict to $L^2(B_1)$, then $E_{B_1}$ is the identity,
showing that $1 \in A_{10}E_{B_0}A_{10}$ on $L^2(B_1)$.  Also
we have already shown that the trace on $B_1$ is the Markov
trace for the inclusion $B_0 \subset B_1$.
So the lower quadrilateral is symmetric Markov, and we can obtain
the subfactor $N \subset P$ by iterating the basic construction
on it.

We conclude that $N \subset P$ and $P \subset M$ are both commuting-square
subfactors, generated by 
$$\commsq{A_{00}}{A_{10}}{B_0}{B_1}$$ and
$$\commsq{B_0}{B_1}{A_{01}}{A_{11}}$$
respectively.
\end{proof}

In~\cite{ZL} the author showed that if a commuting-square subfactor
$N \subset M$ has an intermediate subfactor $P$, then intermediate
algebras
$B_0 \subset B_1$ exist, with upper and lower symmetric commuting squares
as above.  The above theorem may be thought of as
the converse of this result.

Now suppose that
$A_{01}=B_1^H$ and $A_{10}$ is isomorphic to $B_0 \rtimes K$.
Then our subfactor is of the
type described in~\cite{BH}.  In order to find the principal and dual principal graphs,
it is therefore sufficient to find the group generated by $H$ and $K$ in $Out P$.

\begin{thm}
Let $P$ be the $II_1$ factor obtained by iterating the basic construction
on an inclusion of finite-dimensional $C^*$-algebras $B_0 \subset B_1$.
Let $H$ and $K$ be finite groups with outer actions
on $P$, with both actions compatible with the tower of the $B_i$'s.
Let $\rho$ be the representation of $H \ast K$ obtained
by combining these actions.
Then $G = H \ast K / Int P$ may be computed
by considering only $\rho|_{B_1}$.
\end{thm}
\begin{proof}
To find $G=H \ast K/Int P$, it is sufficient to be able to
determine whether $\rho_w$ is outer for an arbitary word $w \in H \ast K$.
But since $H$ and $K$ map into the group of compatible automorphisms,
$\rho_w$ is compatible as well.  It follows that
$\rho_w$ is inner if and only if
$\rho_w|_{B_1} = Ad u|_{B_1}$ for
some unitary $u \in B_0$.
\end{proof}

This can be computed rapidly for any
particular $w$, assuming that $B_1$ is a reasonable size.  If we have a bit
more information about the structure of $G$ it may only be necessary to evaluate
a few thousand words, or even fewer.  In some cases further
simplifications occur, and this computation can be done by hand.

\subsection{Hadamard Subfactors of Depth 2}

A Hadamard matrix is a matrix with orthogonal columns whose entries are all
$\pm 1$~\cite{Ho}.
For our purposes, such matrices are incorrectly scaled:
we will define a complex Hadamard matrix to be an $n \times n$ unitary
matrix whose entries all have the same complex modulus, namely
$n^{-1/2}$.  If $H$
is an $n$ by $n$ complex Hadamard matrix, then from~\cite{JoP}
it is the biunitary
connection for a commuting square of the form
$$\commsq{\mathbb{C}}{H\mathbb{C}^nH^*}{\mathbb{C}^n}{M_n(\mathbb{C})}$$
Likewise if such a quadrilateral commutes, then $H$ must be a Hadamard matrix.

These are Hadamard commuting squares.  They are symmetric and connected,
and so with their Markov trace they give subfactors via
iteration of the basic construction~\cite{JS}.

\begin{dfn}
A \begin{bf}Hadamard subfactor\end{bf} is a subfactor obtained by
iterating the basic construction on the commuting square coming
from a complex Hadamard matrix.
\end{dfn}

Two Hadamard commuting squares are isomorphic if their matrices
are Hadamard equivalent, i.e., if
the matrices
can be obtained from each other by the operations of permuting rows
and columns, and multiplying rows and columns by scalars of modulus 1~\cite{JoP}.
In this case the corresponding Hadamard subfactors are the same.  The
index of a Hadamard subfactor is equal to the size of the matrix.

If $G$ is a finite abelian group, with
$|G|=n$, then its left regular representation
on $l^{\infty}(G)$ gives rise to the commuting square
$$\commsq{\mathbb{C}}{l^{\infty}(G)}{\mathbb{C}[G]}{M_n(\mathbb{C})}$$
by taking $M_n(\mathbb{C}) = l^{\infty}(G) \rtimes G$.

Any two maximal abelian subalgebras of $M_n(\mathbb{C})$ are unitarily
equivalent, so there exists $H_G \in M_n(\mathbb{C})$ with
$Ad H_G(\mathbb{C}[G])=l^{\infty}(G)$.  Since the above square commutes,
$H_G$ must be a complex Hadamard matrix.
We construct $H_G$ as follows.

For an abelian group $G$, $Hom(G,\mathbb{C})$ is isomorphic to $G$.
Specifically,
$G$ has $n$ 1-dimensional representations $\{\rho_g\}$, with
$\rho_g(x)\rho_h(x)=\rho_{gh}(x)$ for $g,h,x \in G$.
Indexing the rows and columns by elements of $G$, we then
let $(H_G)_{ij}=\rho_j(i)$.  This is the discrete Fourier transform
of the group $G$~\cite{Ho}, and is known as the Fourier
matrix when $G$ is cyclic.

\begin{thm}
The discrete Fourier transform $H_G$ of a finite abelian group $G$
gives the commuting square
$$\commsq{\mathbb{C}}{l^{\infty}(G)}{\mathbb{C}[G]}{M_n(\mathbb{C})}$$
and the corresponding Hadamard subfactor is $R \subset R \rtimes G$.
\end{thm}
\begin{proof}
Let $\mathbb{C}[G]$ to be the diagonal subalgebra
of $M_n{\mathbb{C}}$, spanned by minimal projections
$\{e_x|x \in G\}$.
In this case we may take $u_g \in \mathbb{C}[G]$ to be
$\sum_{x \in G} \rho_g(x) e_x$,
where $\rho$ is as above.
The minimal projections
$\{f_x|x \in G\}$ of $l^{\infty}(G)$ are then given as
$f_x = H e_x H^*$, $H=H_G$.  To show that we have found the correct connection,
we should have the $u_g$'s acting on the $f_x$'s via permutation.

We compute the entries of $f_x$ as
$(f_x)_{ij} = H_{ix} (H^*)_{xj} = H_{ix}\overline{H_{jx}}$.  From
the definition of $H_G$, this is $\rho_x(i)\overline{\rho_x(j)}$.
Conjugating by the diagonal unitary
$u_g$ gives $(u_g f_x u_g^*)_{ij}=
\rho_g(i)\rho_x(i)\overline{\rho_x(j) \rho_g(j)}$.  The group
is abelian, so by the properties of $\rho$ this is
$\rho_{gx}(i)\overline{\rho_{gx}(j)}=(f_{gx})_{ij}$.  So the
$u_g$'s act by the regular representation on $H\mathbb{C}[G]H^*$,
and we have indeeed found the connection for the group type
Hadamard commuting square.           

Let $M_0 \subset M_1$ be the Hadamard subfactor induced by this
commuting square.  Then each $u_g \subset \mathbb{C}[G]$ normalizes
$l^{\infty}(G)$ and commutes with the horizontal Jones projections,
hence normalizes $M_0$. $Ad u_g$ is thus a compatible automorphism
of $M_0$, and from theorem 2.3 is outer for
any $g \neq 1$ (in this case we have $B_0 = \mathbb{C}$, so
any nontrivial compatible automorphism is outer).

This means that $\{ \{u_g\}, M_0\}''$
is isomorphic to $M_0 \rtimes G$.  The $u_g$'s span $\mathbb{C}[G]$,
and $M_1$ is generated by $\mathbb{C}[G]$ and $M_0$, so in fact
$M_1 = M_0 \rtimes G$.
\end{proof}

\subsection{Tensor Product of Hadamard Matrices}

Let $H_1$ and $H_2$ be complex
Hadamard matrices, of respective
sizes $m$ and $n$.  Then their tensor product $H=H_1 \otimes H_2$
is unitary.  If we take $i, k \in \{1...m\}$, $j,l \in \{1...n\}$,
then the matrix entry $H_{ij,kl}$ is equal to
$(H_1)_{ik}(H_2)_{jl}$.  Since $H_1$ and $H_2$ are Hadamard, the
complex modulus of this entry does not depend on $i$, $j$, $k$, $l$,
and $H$ is Hadamard as well.
We will use the structure of this
tensor product to construct an intermediate subfactor
of the Hadamard subfactor generated by $H$.

Let $\Delta_m=\mathbb{C}^m$, $\Delta_n=\mathbb{C}^n$
be the diagonal subalgebras of the appropriate
matrix algebras.
We may naturally apply the operator
$Ad H$ to $\Delta_m \otimes \Delta_n$,
and in fact to the entire tower of algebras
$\mathbb{C} \subset \Delta_m \otimes 1 \subset \Delta_m \otimes \Delta_n
\subset  M_m(\mathbb{C}) \otimes \Delta_n \subset M_m(\mathbb{C})
\otimes M_n(\mathbb{C})$
,
giving the following diagram:

$$\begin{array}{ccccc}
\Delta_m \otimes \Delta_n & \subset & M_m(\mathbb{C}) \otimes \Delta_n& \subset &
M_m(\mathbb{C})\otimes M_n (\mathbb{C}) \\
\cup& &\cup& &\cup\\
\Delta_m \otimes 1 & \subset &M_m(\mathbb{C}) \otimes 1& \subset &
H(M_m (\mathbb{C}) \otimes \Delta_n)H^* \\
\cup& &\cup& &\cup\\
\mathbb{C} & \subset & H(\Delta_m \otimes 1)H^* & \subset &
H(\Delta_m \otimes \Delta_n)H^*\\
\end{array}$$

Applying $Ad H$ to each appropriate algebra gives

$$\begin{array}{ccccc}
\Delta_m \otimes \Delta_n & \subset &M_m(\mathbb{C}) \Delta_n& \subset &
M_m(\mathbb{C})\otimes M_n (\mathbb{C}) \\
\cup& &\cup& &\cup\\
\Delta_m \otimes 1 & \subset &M_m(\mathbb{C}) \otimes 1& \subset &
M_m (\mathbb{C}) \otimes H_2(\Delta_n)H_2^* \\
\cup& &\cup& &\cup\\
\mathbb{C} & \subset & (H_1 \Delta_m H_1^*) \otimes 1 & \subset &
(H_1 \Delta_m H_1^*) \otimes (H_2 \Delta_n H_2^*)\\
\end{array}$$

 Next we consider
$$E_{M_m(\mathbb{C}) \otimes \Delta_n}(\Delta_m \otimes \Delta_n)=
\Delta_m \otimes E_{H_2 \Delta_n H_2^*}(\Delta_n)$$
Since $$\hadsq{H_2}{n}$$ commutes, the above conditional expectation
is $\Delta_m \otimes 1$, implying that
$$\commsq{\Delta_m \otimes 1}{H(M_m(\mathbb{C}) \otimes \Delta_n)H^*}
{\Delta_m \otimes \Delta_n}
{M_m(\mathbb{C}) \otimes M_n(\mathbb{C})}$$
commutes as well.  Furthermore, $\Delta_m \otimes 1 \subset
H(M_m(\mathbb{C}) \otimes \Delta_n)H^*$ is a connected inclusion,
since $(\Delta_m \otimes 1) \cap (M_m(\mathbb{C}) \otimes 1)' = \mathbb{C}$.

From theorem 3.1, we then
have an intermediate subfactor $P$, obtained by iterating
the basic construction on
$$\Delta_m \otimes 1 \subset H(M_m(\mathbb{C}) \otimes \Delta_n)H^*$$

We will refer frequently to these algebras later, so
let us define 
$B_0=\Delta_m \otimes 1$,
$B_1 =H(M_m(\mathbb{C})\otimes \Delta_n)H^* =
M_m(\mathbb{C}) \otimes (H_2 \Delta_n H_2^*) $

\begin{thm}
Let $H_1 = H_H^*$ and $H_2 = H_K$, for finite abelian
groups $H$, $K$.  Then the subfactor obtained from the Hadamard
matrix
$H_1 \otimes H_2$ is of Bisch-Haagerup type.
\end{thm}
\begin{proof}
Let $\Delta_{m,n} \subset M_{m,n}(\mathbb{C})$ be as above,
with the intermediate subfactor $P$ obtained as above from the inclusion
$B_0 \subset B_1$.
From the definition of the group Hadamard matrix, we
have $\{v_h\} \subset \Delta_m$ 
such that $H_1 v_h H_1^*$ acts via
the left regular representation of $H$ on $\Delta_m=l^{\infty}(H)$.
Then $u_h=(H_1 v_h H_1^*) \otimes 1$ is an
element of $M_m(\mathbb{C}) \otimes (H_2 \Delta_n H_2^*)$
which normalizes $\Delta_m \otimes 1$.  $ Ad u_h$
therefore induces a compatible action on $P$.  This action
is outer for $h \neq 1$,
since any element of $\Delta_m \otimes 1$
acts trivially by conjugation on $\Delta_m \otimes 1$,
while $u_h$ acts on this algebra by nontrivial permutation of the minimal
projections.
We therefore obtain an outer compatible action of $H$ on
$P$, induced by $Ad (u_h)$.

Since $u_h \in H(\Delta_m \otimes \Delta_n)H^*$,
and this algebra is abelian, $P^H$ contains
$H(\Delta_m \otimes \Delta_n)H^*$.  The action of $H$
is compatible, so $e_i \in P^H$ as well, implying that
$N \subset P^H$. The index $[P:P^H]=m=[P:N]$, so in fact $N=P^H$
with the action of $H$ induced by $Ad u_h$.

Likewise, we have $v_k \in \Delta_n$ acting via the left regular
representation of $K$ on
$H_2 \Delta_n H_2^*$.  Taking
$u_k = 1 \otimes v_k$, this gives
$\{u_k\}$ as elements of $\Delta_m \otimes \Delta_n$,
which act via conjugation on
$M_m(\mathbb{C}) \otimes(H_2 \Delta_n H_2^*)$ and normalize
(in fact act trivially on) $\Delta_m \otimes 1$.  $u_k$ commutes
with the horizontal Jones projections $\{e_i\}$, so
$Ad u_k$ fixes these projections and acts
compatibly on $P$.  $Ad u_k$ acts
nontrivially on the center of $M_m(\mathbb{C}) \otimes (H_2 \Delta_n H_2^*)$,
so the induced action is outer.
This means that we have the factor
$P \rtimes K$ embedded in $M$ as the algebra generated
by $P$ and $\{u_k\}$, with the action given by $Ad (u_k)$.
Since $[M:P]=n=[P \rtimes K:P]$, in fact $M = P \rtimes K$.

Therefore the subfactor is $P^H \subset P \rtimes K$,
and it may be analyzed using~\cite{BH}.
\end{proof}

Define a map $a$ from $H \oplus K$ to $M_m \otimes M_n$
by $a(g)=u_h $, $a(k)=u_k$.  Then
$a$ induces an outer action $\alpha$ of $H \oplus K$ on
$P$.
This action extends to the free product, and
$N \subset P \subset M$ is equal to $P^H \subset P \otimes K$
where the action of both $H$ and $K$ are given by $\alpha$.

We may now use the methods of Bisch and Haagerup to find the
principal graph.  From~\cite{BH}, we need only find the group
$G$ generated by $H$ and $K$ in the group $Out P=Aut P/Int P$.
In this case, $(H_1 \Delta_m H_1^*)\otimes 1$ commutes with
$1 \otimes \Delta_n$,
implying that $H$ and $K$ commute in $Out P$
and $G = H \oplus K$.  In
this case we know that the composite subfactor $P^H \subset
P \subset P\rtimes K$ is depth 2, and is equal to
$R^{H \oplus K} \subset R$.  So the tensor product construction
 described
above does not give
any new examples.

\subsection{The Twisted Tensor Product}

If $H_1$ and $H_2$ are Hadamard, then as in the previous
section their tensor product $H$ is Hadamard as well, with
$H_{ij,kl}=(H_1)_{ik}(H_2)_{jl}$.  The twisted tensor
product is the matrix
$H_{ij,kl}=(H_1)_{ik}(H_2)_{jl}\lambda_{il}$, where
each $\lambda_{il}$
is an arbitrary complex number of modulus 1.  
Each matrix entry of this twisted tensor product has
the same complex modulus, namely $(nm)^{-\frac{1}{2}}$.
To see that $H$ is still unitary,
we take $T$ to be the unitary element of the diagonal algebra
$\Delta_m \otimes \Delta_n$ with components $T_{gh,gh}=\lambda_{gh}$.
Then $H$ is equal to the matrix product
$(1 \otimes H_2)T(H_1 \otimes 1)$, which is unitary.

\begin{thm}
Let $H_1 = H_H^*$ and $H_2 =H_K$, for finite abelian groups $H$
and $K$.  Let $T \in \mathbb{T}^{|H||K|}$ be a twist.  Then
the twisted tensor product 
$H=(1 \otimes H_2)T(H_1 \otimes 1)$ induces a Hadamard subfactor
of Bisch-Haagerup type.
\end{thm}
\begin{proof}
Applying the Hadamard matrix $H$ to the tower of algebras
$$\mathbb{C} \subset \Delta_m \otimes 1 \subset \Delta_m \otimes \Delta_n
\subset  M_m(\mathbb{C}) \otimes \Delta_n \subset M_m(\mathbb{C})
\otimes M_n(\mathbb{C})$$
gives the following diagram:

$$\begin{array}{ccccc}
\Delta_m \otimes \Delta_n & \subset &M_m(\mathbb{C})\otimes \Delta_n & \subset &
M_m(\mathbb{C})\otimes M_n(\mathbb{C}) \\
 & &\cup& & \\
\cup& &M_m(\mathbb{C}) \otimes 1& &\cup\\
 &\subset& &\subset& \\
\Delta_m \otimes 1 & & & & H(M_m (\mathbb{C}) \otimes \Delta_n)H^* \\
 &\subset& &\subset& \\
\cup& &H(M_m(\mathbb{C}) \otimes 1)H^*& &\cup\\
 & &\cup& & \\
\mathbb{C} & \subset & H(\Delta_m \otimes 1)H^* & \subset &
H(\Delta_m \otimes \Delta_n)H^*\\
\end{array}$$

There are now two inclusions which are not obviously correct,
namely $\Delta_m \otimes 1\subset H(M_m(\mathbb{C}) \otimes 1)H^*$
and $M_m(\mathbb{C}) \otimes 1 \subset H( M_m(\mathbb{C}) \otimes \Delta_n)H^*$.

Note first that $H_1 \otimes 1$ normalizes $(M_m(\mathbb{C}) \otimes 1$,
so $H(M_m(\mathbb{C}) \otimes 1)H^* =
(1 \otimes H_2)T(M_m(\mathbb{C} \otimes 1)T^*(1 \otimes H_2)$.
$\Delta_m \otimes 1$ is contained
in $M_m(\mathbb{C}) \otimes 1$ and commutes with both $T$ and $1 \otimes H_2$,
so $\Delta_m \otimes 1 \subset H (M_m(\mathbb{C}) \otimes 1)H^*$.

Now we consider $H(M_m(\mathbb{C}) \otimes \Delta_n)H^*$.
$M_m(\mathbb{C}) \otimes \Delta_n$ is
the commutant in $M_m(\mathbb{C}) \otimes M_n(\mathbb{C})$
of $1 \otimes \Delta_n$, so
$H(M_m(\mathbb{C}) \otimes \Delta_n)H^* = (H(1 \otimes \Delta_n)H^*)'$ as well.
But $H_1 \otimes 1$ and $T$ commute with $ 1 \otimes \Delta_n$, so
this is the commutant of
$(1 \otimes H_2)(1 \otimes \Delta_n)(1 \otimes H_2^*)$.  We conclude
that  $$ H(M_m(\mathbb{C}) \otimes \Delta_n)H^* =
M_m(\mathbb{C}) \otimes (H_2 \Delta_n H_2^*)$$
which includes $M_m(\mathbb{C}) \otimes 1$.

So all inclusions in the above diagram are correct.
As before,
the quadrilateral
$$\commsq{\Delta_m \otimes 1}{H(M_m(\mathbb{C}) \otimes \Delta_n)H^*}
{\Delta_m \otimes \Delta_n}{M_m(\mathbb{C}) \otimes M_n(\mathbb{C})}$$
commutes, since the algebra $H(M_m(\mathbb{C}) \otimes \Delta_n)H^*$
is not affected by the choice of $T$.  All horizontal
inclusions are connected, so from theorem
3.1 we again have
an intermediate subfactor $P$, obtained from iterating the basic
construction on $$B_0 = \Delta_m \otimes 1
\subset M_m(\mathbb{C}) \otimes (H_2 \Delta_n H_2^*)=B_1$$

We will now describe the action of $H$ and $K$ (and their free product)
induced on $P$ by the twisted tensor product construction.

Exactly as in section 3.4, we have $v_k \in \Delta_m$,
$u_k = 1 \otimes v_k \in \Delta_m \otimes \Delta n$
which normalize $B_1$ and commute with $B_0$.  These act via the left regular
representation on $Z(B_1)$ and induce an outer compatible action of
$K$ on $P$.  The same index argument as before shows that
$M = P \rtimes K$.

The action of $H$ is slightly different.  Again we have
$v_h \subset \Delta_m$ with $H_1 v_h H_1^*$ acting via the left
regular representation on $\Delta_m$.  Now, however, we define
$u_h$ by $u_h=H(v_h \otimes 1) H^*$, which reduces to the earlier
definition of $1 \otimes (H_2 v_h H_2^*)$
if $T$ is the identity.  As before, $u_h$ is contained in the abelian
algebra $H (\Delta_m \otimes \Delta_n)H^* \subset N$ and commutes
with the $e_i$'s.  Therefore, an index argument again gives
$P^H = N$, where the action of $H$ is defined by taking the
compatible automorphisms induced by the $Ad u_h$'s.
\end{proof}

\subsection{The action of $H$ and $K$}

We now give a bit more detail on
the role of the twist in producing the actions of $H$ and $K$
for a twisted tensor product Hadamard subfactor,
with all notation as in the previous section.

Let $\tilde{T}=(1 \otimes H_2)T(1 \otimes H_2^*)$.
Then $\tilde{T} \in B_0' \cap B_1$.
Since $H_1 \otimes 1$ and $1 \otimes H_2$ normalize
$M_m(\mathbb{C}) \otimes 1$,
$$Ad (\tilde{T})(M_m(\mathbb{C}) \otimes 1)
= Ad(\tilde{T}(H_1 \otimes H_2))(M_m(\mathbb{C})\otimes 1)
= Ad H(M_m(\mathbb{C}) \otimes 1)$$  So $\tilde{T}$ gives
the difference between these two matrix algebras.

We have $u_h \in H(M_m(\mathbb{C}) \otimes 1)H^*$, while
$1 \otimes u_k$ commutes with $M_m \otimes 1$.  It follows that
if $\tilde{T}$ normalizes $M_m \otimes 1$, then the two
group actions commute with each other and we are again in the depth
2 case.   Conversely, for general $T$, the two algebras
$M_m \otimes 1$ and $Ad T(M_m \otimes 1)$
are different, and we will not expect the actions to commute.

As before, let $a(h)=(H_1 v_h H_1^*) \otimes 1$,
$a(k)=1 \otimes v_k$.  This still provides an induced action $\alpha$
of $H \oplus K$
on $P$, but $a(h)$ does not commute with $N$, so this direct sum no longer
describes the structure of the subfactor.  Instead, note
that $u_h = H(v_h \otimes 1)H^*=
Ad \tilde{T}(a(h))$, since $1 \otimes H_2$ commutes
with $a(h)$.
$\tilde{T}$ itself induces a compatible automorphism $\tau$ on $P$, since
$Ad \tilde{T}$ normalizes $B_1$ and $B_0$.  It follows from the
properties of compatible automorphisms that
$u_h$ induces the compatible automorphism $\tau \alpha_h \tau^{-1}$.

This allows us to describe the correct action $\beta$ of the free
product $H \ast K$.  Let $b(h)=\tilde{T}a(h)\tilde{T}^*$, $b(k)=a(k)$.
$b$ induces the compatible action $\beta$, defined
by $\beta_h=\tau \alpha_h \tau^{-1}$, $\beta_k=\alpha_k$.
We have $N=P^H$, $M=P \rtimes K$ for this action, so the properties
of the subfactor
$N \subset M$ may be determined by examining $b(H \ast K)$.
For general $T$, $\tau \alpha_h \tau^{-1}$ will not commute
with $\alpha_k$, even as elements of $Out P$.  This construction will therefore
provide nontrival new examples.

There are still significant restrictions on the action $\beta$
of $H \ast K$.  $Ad b(h)$ acts on the center of $\Delta_m \otimes 1$
via permutation, but acts trivially on $1 \otimes (H_2 \Delta_n H_2^*)$;
$Ad b(k)$ does the opposite.
Any word $w$ in $H\ast K$ whose restriction to $B_1=M_m(\mathbb{C})
\otimes (H_2 \Delta_n H_2^*) \subset P$
permutes the basis projections of $B_0' \cap B_1 =
\Delta_m \otimes
(H_2 \Delta H_2^*)$
nontrivially must be outer, since $Ad u$ acts trivially on this set for
$u \in B_0$.  Let $\epsilon$ be the abelianization map from
$H \ast K$ to $H \oplus K$; it follows that if $\epsilon(w)$ is
not the identity, then $w$ is outer.

We conclude that the inner subgroup $H \ast K \cap Int$ is contained in the
first commutator subgroup $N_1$ of $H \ast K$, where $N_1$ is the
group generated by elements of the form $hkh^{-1}k^{-1}$.

Now let $x$ be in $N_1$,  $b(x) \in M_m(\mathbb{C}) \otimes M_n(\mathbb{C})$.
$Ad b(x)$ acts trivially on $\Delta_m \otimes (H_2 \Delta_n H_2^*)$
since the induced permutations of $\Delta_m \otimes 1$ and
$1 \otimes (H_2 \Delta_n H_2^*)$
are trivial.  This means that $b(x)$ commutes with this algebra.
But $\Delta_m \otimes (H_2 \Delta_n H_2^*)$
is maximal abelian in
$M_m(\mathbb{C}) \otimes M_n(\mathbb{C})$,
so $b(x)$ must be contained in $B_0' \cap B_1$. Since
this algebra is abelian,
$b(x)$ must commute with
$b(n)$ for every
other element $n$ of $N_1$.  This means that the induced actions on $P$
commute as well, and $b(N_1)$ is abelian.
Taking $G = H \ast K / Int$, and $N = N_1/Int$ the first commutator
subgroup of $G$, it follows that $N$ is abelian as well.

Each element of any group with the above properties may be written uniquely
as $hkn$, $h \in H$, $k \in K$, $n \in N$.  We
therefore write $G=H \ast K / Int = HKN$.
$N$ is an abelian group generated by the $(|H|-1)(|K|-1)$ elements
of the form $hkh^{-1}k^{-1}$, $h \neq 1 \in H$, $k \neq 1 \in K$.
$N$ is a normal subgroup of $H \ast K$.
We may determine $N=N_1/Int$, and therefore
$G$ itself, by determining which elements
$n$ of $N_1$ have $\beta_n$ outer. 

From section 2.7, $\beta_n$ is inner if and only if 
$Ad b(n)|_{B_1}=Ad u|_{B_1}$ for some unitary
$u \in B_0$, with $B_0$ and $B_1$ as above.  This will be true if and only
if $b(n)=uv$, where $u$ and $v$ are unitary elements of
$B_0$, $Z(B_1)$ respectively.
This allows us to readily determine the order in $Out P$
of each generator of $N_1$,
and hence the structure of the group $N_1 / Int$.

To find $H \ast K / Int$, we
must also know how $H$ and $K$ act on $N$; i.e.,
$hnh^{-1}$ and $knk^{-1}$ for $h \in H$, $k \in K$, $n \in N$.
To find these, we consider $B_0' \cap B_1 =
l^{\infty}(H) \otimes l^{\infty}(K)= l^{\infty}(H \oplus K)$.
In this representation, two elements of $B_0' \cap B_1$
induce the same local automorphism up to inner perturbation if they differ by
a unitary $u$ with coordinates $u(h,k)=f(h)g(k)$, where $f$
and $g$ are functions from $H$ and $K$ respectively to the complex
scalars of modulus 1.
We may therefore put each element of $N$ in a unique standard form,
with $n_{std}(1,k)=1=n_{std}(h,1)$
With this description of $N$, $Ad b(H)$ and $Ad b(K)$ act
via the left regular representation $\rho$ on the appropriate component.
We know that $hnh^{-1}$ must be equivalent to some
$n' \in N$, and we can readily determine which one by
putting $hnh^{-1}$ in standard form.  The same holds
true for the action of $K$.

The above description of $N$ provides a particularly good
way of writing the generators.  For $x=b(hkh^{-1}k^{-1}) \in b(N)$,
we know that $$x=\tilde{T}a(h)\tilde{T}^*a(h)\tilde{T}a(g^{-1})
\tilde{T}^*a(h^{-1})$$
$\tilde{T}$ is itself an element of $B_0' \cap B_1 = l^{\infty}(H\oplus K)$.
$Ad a(h)$ and $Ad a(k)$  act on
$ l^{\infty}(H \oplus K)$ via the left regular representation $\rho$, so we have
$$x=\tilde{T} \rho_h(\tilde{T}^*) \rho_{hk}(\tilde{T})
\rho_{hkh^{-1}}(\tilde{T}^*)a(h)a(k)a(h^{-1})a(k^{-1})$$
$[\rho_h, \rho_k]=0=[a(h),a(k)]$, so this is just
$T \rho_h(T*) \rho_k(T^*) \rho_{hk}(T)$.  This gives
us $x=b(hkh^{-1}k^{-1})$ as an element of $B_0' \cap B_1$.

This allows the complete
computation of $G= H\ast K / Int = HKN$ for any
twisted tensor product Hadamard subfactor.  Using the methods of~\cite{BH},
we may then obtain the principal graphs.

Since $B_0$ is fixed by $K$,
multiplying $\tilde{T}$ by an inner $z \in B_0$ 
will not affect any of the generators of $N$:
the change to $b(hkh^{-1}k^{-1})$ will be multiplication by
$$z \rho_h(z^*) \rho_k(z^*) \rho_{hk}(z) = z \rho_h(z^*) z^* \rho_h(z) = 1$$
for any $h \in H$, $k \in K$.
The same is true of any perturbation coming from $Z(B_1)$.

We may therefore put $\tilde{T}$ itself in standard form without
affecting the action of $\beta(N)$.  The size of the group $G$
is determined by $\beta|_N$; as we will see in section 5.1,
the 3-cocycle obstruction associated with the action of $G$
is as well.  The group $G$ and its 3-cocycle determine
the standard invariant of a Bisch-Haagerup subfactor(see section 4,
~\cite{BDG2}),
so we only
need to consider twists in standard form.
This will give us a better idea of the size
of the space of examples obtained from this construction.

\section{Classification of Bisch-Haagerup subfactors}

\subsection{Introduction}

Before giving specific examples, we will discuss the classification
up to subfactor isomorphism of the subfactors obtained from this
twisted tensor product, in the finite depth case.
Let $M^H \subset M \rtimes K$ be a hyperfinite $II_1$
subfactor.  The principal graph is determined entirely from the group
$G=H\ast K/Int P$, as is the dual principal graph~\cite{BH}.  Throughout
this section, we will always require the subfactor to be finite
depth; equivalently, we will always take this group $G$ to be finite.

$G$ is contained
in $Out P$, and does
not always lift to $Aut P$.  
For a finite group, the obstruction to this lift
is an element $\omega$ of $H^3(G)$.  Two representations
$\rho, \sigma$ of $G$ in $Out P$ are said to be outer conjugate
if there exists $\alpha \in Out P$ with
$\alpha\sigma_g\alpha^{-1}$ equivalent to $\rho_g$
in $Out P$ for all $g \in G$.  From \cite{JoT}, this will occur
if and only if the associated 3-cocycles are the same.

Bisch, Das, and Ghosh have recently shown that $H$, $K$, $G$, and
$\omega$ determine the planar algebra (or equivalently, the
standard invariant) of the subfactor, in the case $\omega=0$~\cite{BDG}
and more recently for general $\omega$~\cite{BDG2}.  If $G$
is finite, then the subfactor is finite depth, and therefore
strongly amenable.
Such are classified up
to isomorphism by
their standard invariants~\cite{PoA}, so this implies
that the outer conjugacy class of $G$ determines the isomorphism class
of the subfactor in this case.

We provide a more elementary proof of this result below (which requires
some additional assumptions on the actions) and also
discuss the converse implication.

\subsection{Invariants of Group Actions}

In this section we summarize some results from~\cite{JoT}, which
we will use to classify our group actions.

Let $\tilde{G}$ be a finite group acting via $\rho$ on the hyperfinite $II_1$
factor $R$, with inner subgroup $S$.  Let $S$
be implemented by unitaries $u_S$, i.e. $\rho_s = Ad u_s$ for $s \in S$.
For $g \in \tilde{G}$, $s, s' \in S$,
we have $\alpha_g(u_{g^{-1}sg})=\lambda(g,s)u_s$,
$u_s u_{s'}=\mu(s,s')u_{ss'}$.  Here $\lambda$
is a function from $G \times N$ to the complex scalars
of modulus 1, and $\mu$ is a similar function on $N \times N$.

Jones defines the characteristic invariant
$\Lambda_{\tilde{G},S}$ as the set of all such pairs
$(\lambda, \mu)$ which are allowable, in the sense that they
can actually arise from some action of $\tilde{G}$ on $R$.
Let $L$ be the left regular representation
of $\tilde{G}$ on $l^{\infty}(\tilde{G})$.  The following definition
is from~\cite{JoT}:

\begin{dfn}
Two actions $\rho$, $\sigma$ of
$\tilde{G}$ with inner subgroup $S$
are \begin{bf}stably conjugate\end{bf} if
$\rho \otimes Ad L$ is conjugate to $\sigma \otimes Ad L$.
\end{dfn}

Two actions of $\tilde{G}$ are stably conjugate if and only if their
characteristic invariants (and inner subgroups)
are the same.

Jones also defines the inner invariant, which is determined
from the restriction of the trace to $\mathbb{C}[S]$.  Two
stably conjugate actions with the same inner invariant
are actually conjugate.

An action of $\tilde{G}$ on R provides a representation of
the kernel $G=\tilde{G}/S$ in $Out R$.  This representation
lifts to an action of $G$ on $R$ if and only if
the associated obstruction $\omega \in H^3(G)$ is zero.
$\omega$ may be computed from the characteristic invariant.
In fact, from~\cite{JoT} there is an exact sequence
$H^2(\tilde{G}) \rightarrow \Lambda(\tilde{G},S) \rightarrow H^3(G)$.

For some particular $(\lambda, \mu)$ it may be difficult
to determine if the corresponding 3-cocycle comes from
a coboundary.  One useful method is the restriction to
cyclic subgroups.  If $(\lambda, \mu)$ comes
from some $\epsilon \in H^2(\tilde{G})$, then the restriction
to any cyclic subgroup $C \subset \tilde{G}$ also comes from a 2-cocycle,
and therefore maps to zero in $H^3$.  Therefore
if $(\lambda, \mu)|_C$ maps to a nonzero element
of $H^3(C)$, the entire action has nontrivial 3-cocycle.
Since $H^2(C)=0$ for cyclic groups, this will be true
if $(\lambda, \mu)_C$ itself is nonzero.

\subsection{General Assumptions}

Let $\rho$ and $\sigma$ be two actions of the free product
$H \ast K$ on hyperfinite $II_1$ factors $M, P$ respectively,
i.e. maps from $H \ast K$ to $Aut M$, $Aut P$.
For both actions, we will take the restrictions to $H$ and $K$
to be outer.
Let $\tilde{G} = H \ast K / (\ker(\rho) \cap \ker(\sigma))$,
which we will always take to be finite.
Here $\ker(\rho)$ is the subset of $H \ast K$ which
$\rho$ sends to the identity map, and likewise for $\sigma$.
We will then have actions of
$\tilde{G}$ on $M$ and $P$.
Let $S$ be the inner subgroup $\tilde{G} \cap Int $
of $\tilde{G}$,
which we take to be the same for the actions
$\rho$ and $\sigma$.
Then $G = \tilde{G} / S = H \ast K / Int$
is the same for both actions as well.
Note that $\rho$ and $\sigma$ are determined by
their respective restrictions to $H \cup K$.

We will consider the subfactors $M^H \subset M \rtimes K$,
$P^H \subset P \rtimes K$.  The action of $K$
will be implemented in $M \rtimes K$ by unitaries $u_k$ and
in $P \rtimes K$ by unitaries $v_k$.

A unitary cocycle for an action $\rho$ of $\tilde{G}$ is a set
of unitaries $\{w_g|g \in \tilde{G}\}$ with $w_g \rho_g(w_h)=w_{gh}$.
This condition is sufficient for $Ad w_g \circ \rho_g$
to itself be an action, so such a cocycle may be thought
of as an inner perturbation of the original action.
However, since $Ad \lambda w_g = Ad w_g$ for $\lambda$ any complex
scalar of modulus 1,
the above condition is stronger than necessary.

Let $\epsilon$ be a
function from $\tilde{G} \times \tilde{G}$ to the complex numbers of modulus 1.
We may then
take $\{w_g\}$ with $w_g \rho_g(w_h)=\epsilon(g,h)w_{gh}$.
For such $\{w_g\}$, $Ad w_g \circ \rho_g$ is still an action.
Associativity of multiplication gives $\epsilon(x,yz)\epsilon(y,z)=
\epsilon(x,y)\epsilon(xy,z)$.  These are the cocycle relations
defining elements of $H^2(\tilde{G})$~\cite{JoT},
so we may identify
$\epsilon$ as a scalar 2-cocycle in this space.
We will call $\{w_g\}$ a twisted unitary cocycle with
associated scalar cocycle $\epsilon$.

We have $S \subset \tilde{G}$ a normal subgroup.
Due to linearity of the map from $H^2(G)$ to $\Lambda_{G,N}$
in~\cite{JoT},
the characteristic invariant of $Ad w \circ \rho$
is the same as the characteristic invariant of $\rho$
plus the element of $\Lambda_{\tilde{G},S}$ induced by $\epsilon$.

\subsection{Stable Conjugacy Implies Subfactor Isomorphism}

\begin{lem}
Conjugacy of two actions of $G=H \ast K / Int$ with outer
restrictions to $H$ and $K$
implies subfactor isomorphism of the corresponding Bisch-Haagerup
subfactors.
\end{lem}
\begin{proof}
If $\rho$ and $\sigma$ are conjugate via some isomorphism
$\alpha: P \rightarrow M$, then $\alpha$ extends to
$P \rtimes K$ via $\alpha(v_k)=u_k$.  Since
$\alpha \sigma_h \alpha^{-1} = \rho_h$ for $h \in H$, we have
$\alpha(P^H)=M^H$.  This gives isomorphism of subfactors.
\end{proof}
\begin{lem}
Let $\rho$ give outer actions of the finite groups $H$
and $K$ on the hyperfinite $II_1$ factor $M$, with
$\tilde{G} = H \ast K / \ker(\rho)$ a finite group.
Let $L$ be the left regular action of
$\tilde{G}$ on
$l^{\infty}(\tilde{G})=L^2(\tilde{G})$.  Then 
$\rho \otimes Ad L$ gives an action of $H \ast K$ on
$M \otimes \B(L^2(\tilde{G}))$, where $Ad L$ is the adjoint
action of $L$.
In this case
$\rho$ and $\rho \otimes Ad L$ induce the same subfactor.
\end{lem}
\begin{proof}
$\rho$ and $\rho \otimes 1$ have the same characteristic invariants,
where $\rho \otimes 1$ acts on the factor $M \otimes \B(L^2(G))$.
Since $\tr(x \otimes 1) = \tr(x)$, the two actions have
the same inner invariant as well, hence are conjugate~\cite{JoT}.
By the previous lemma
they induce the same subfactor.

$\rho|_H \otimes 1$ is outer, and $1 \otimes Ad L(H)$ is a unitary
cocycle for this action (with no scalar twist). 
From~\cite{JoT}, every such cocycle is a unitary coboundary.
This means
there is some unitary
$w_H \in M \otimes \B(L^2(\tilde{G}))$
such that
$$w_H (\rho_h \otimes 1)(w_H^*) = 1 \otimes L(H)$$
for all $h \in H$, implying that
$$Ad w_H (\rho_h \otimes 1)Ad w_H^* =\rho_h \otimes Ad L(H)$$
Likewise, there is $w_K \in M \otimes \B(L^2(\tilde{G}))$ with
$$Ad w_K (\rho_k \otimes 1)Ad w_k^* =\rho_k \otimes Ad L(k)$$
for all $k \in K$.

Let $\chi$ be the action of $\rho \otimes 1$ on
$M \otimes \B(L^2(\tilde{G}))$, and $\xi$ the action of
$\rho \otimes Ad L$.  
Let the unitaries implementing
the action of $K$ in the respective crossed products
be $\{u_k\}$, $\{v_k\}$.

We define a map $\alpha$ from
$M \rtimes_{\chi} K$ to $M \rtimes_{\xi} K$ by
$\alpha(x)= w_H x w_H^*$ for $x \in M$,
and  
$\alpha(u_k)= w_H w_K^* v_k w_K w_H^*$,

$\alpha|_M$ is clearly an isomorphism.
For $x \in M$, $k \in K$,
we compute
$$\alpha(u_k) \alpha(x)\alpha(u_k^*) =
w_H w_K^* v_k w_K x w_K^* v_k^* w_K w_H^*=
w_H w_K^* \xi_k(w_K x w_K^*) w_K w_H^*$$
Since $Ad w_K \chi_k Ad w_K^* = \xi_k$,
we have $Ad w_K^* \xi_k Ad w_K = \chi_k$,
and this is $w_H\chi_k(x)w_H^* = \alpha(u_k x u_k^*)$.
It follows that
$\alpha$ is an homomorphism on all of $M \rtimes_{\chi} K$;
$\alpha$ is invertible by construction and is therefore
an isomorphism on this factor.

Now we compute $\alpha(M^H)$.  We write $M^{\chi|_H}$, $M^{\xi|_H}$
for the fixed point algebras of the actions of $H$
given by $\chi$, $\xi$ respectively.
$$\alpha(M^{\chi|_H}) = w_H M^{\chi|_H} w_H^*
=M^{Ad w_H \chi|_H Ad w_H^*}=M^{\xi|_H}$$

Therefore $\alpha$ provides an isomorphism between the Bisch-Haagerup
 subfactors
induced by $\chi$ and $\xi$.
\end{proof}

Recall that $\rho$ and $\sigma$ are stably conjugate if
$\rho \otimes Ad L$ is conjugate to $\sigma \otimes Ad L$.
In this case by the previous two lemmas
$\rho \otimes Ad L$ and $\sigma \otimes Ad L$ induce
the same subfactor, and so $\rho$ and $\sigma$ do as well.

\subsection{Outer Conjugacy Implies Subfactor Triplet Isomorphism}

Again let $\rho$ and $\sigma$ act on $M$ and $P$, with
properties as in the previous section.
Now let $\rho$ and $\sigma$ merely be outer conjugate.
That is, there is some isomorphism $\alpha$ from $P$ to $M$
such that for $g \in \tilde{G}$,
$\alpha \sigma_g \alpha^{-1}$ and $\rho_g$
have the same image
in $Out M$, but may differ by some nontrivial inner automorphism.
From~\cite{JoT},
we know that outer conjugacy implies that the actions of $\rho$
and $\sigma$ induce the same 3-cocycle in $H^3(G)$.

\begin{thm}
Let $\rho$ and $\sigma$ be actions of $H \ast K$ on hyperfinite
$II_1$ factors $M$ and
$P$ respectively, with outer restrictions to $H$ and $K$, and
$G = H \ast K / Int$ the same
for both actions.
Let $\tilde{G} = H \ast K / (\ker(\rho) \cap \ker(\sigma))$ be
a finite group.
Assume also that there are homomorphisms $p_H: G \rightarrow H$
and $p_K: G \rightarrow K$ with
$p_H(H)=H$, $p_H(K)=1$, $p_K(K)=K$, $p_K(H)=1$.
Then if $\rho$ and $\sigma$ are outer conjugate.
the induced Bisch-Haagerup subfactors are isomorphic.
\end{thm}
\begin{proof}
 From the
exact sequence of~\cite{JoT}, outer conjugacy implies that
the characteristic invariants of the two actions
differ by the image in $\Lambda_{\tilde{G},S}$ of a 2-cocycle
$\epsilon \in H^2(\tilde{G})$.  In additive notation, $\epsilon$ induces
the characteristic invariant
$(\lambda,\mu)_{\sigma}-(\lambda,mu)_{\rho}$.  We will throughout this
theorem use
additive notation for cohomology groups and characteristic invariants.

We now discuss our assumption on the group.
Let $N_1$ be the first commutator subgroup of $H \ast K$.
The quotient $H \ast K / N$ is $H \oplus K$; the quotient map
sends $H$ to $(H,1)$ and $K$ to $(1,K)$.
Our assumption about the maps $p_H$, $p_K$ is then true if
this map factors through $G$. 
This is equivalent to requiring all words $w \in H \ast K$ with $\rho_w$ inner
(and $\sigma_w$ likewise inner)
to be contained in $N$.

A homomorphism of groups induces a homomorphism of cohomology groups
in the opposite direction.
Therefore we may use $p_H$ and $p_K$ to construct maps
from $H^2(H)$ and $H^2(K)$ into $H^2(G)$. 
We may compose
these maps with the inflation from $H^2(G)$ to $H^2(\tilde{G})$
(obtained from the quotient map $\tilde{G} \rightarrow \tilde{G}/S=G$)
to obtain maps $infl_H: H^2(H) \rightarrow H^2(\tilde{G})$ and
$infl_K: H^2(K) \rightarrow H^2(\tilde{G})$.
The inflation to $H^2(\tilde{G})$ of any 2-cocycle in
$H^2(G)$ has trivial characteristic invariant in
$\Lambda_{\tilde{G},S}$~\cite{JoT},
so the images of $infl_H$ and $infl_K$ do as well.

$p_H$ restricted to $H$ is the identity map, and the same is true
of the quotient map $\tilde{G} \rightarrow G$.
It follows that for $\xi \in H^2(H)$, we have
$infl_H(\xi)|_H = \xi$, and likewise
$infl_K(\xi)|_K = \xi$ for $\xi \in H^2(K)$.

Now we define 2-cocycles on $H$ and $K$ by restricting $\epsilon$
to these subgroups.
Note that since $p_H(K)=\{1\}$,
$infl_H(\epsilon|_H)|_K = infl_H(\epsilon_H)(1,1) = 0$.  The same
argument applies to $infl_K(\epsilon|_K)$.  Therefore
we may define a new cocycle $\xi = \epsilon -
infl_H(\epsilon|_H)-infl_K(\epsilon|_K)$, whose
restriction to $H$ and $K$ is zero.  $\xi$ induces the same
characteristic invariant as $\epsilon$, since it differs
from $\epsilon$ by an inflation of some element
in $H^2(G)$.

We extend our representation to $M \otimes \B(L^2(\tilde{G}))$.
Let $L$ be the action of $\tilde{G}$ on $\B(L^2(\tilde{G}))$
given by $Ad u_g$, where $u_g$ is the left regular representation.
From the previous section, $\rho$ and $\rho \otimes Ad L$ induce
the same subfactor.

We then have a natural system of equivariant projections for $\tilde{G}$,
namely
$e_x \in 1 \otimes \B(L^2(\tilde{G}))$.
These projections partition unity, are fixed
by $\rho \otimes 1$, and have $(\rho_x \otimes Ad L(x))(e_y)=e_{xy}$.

Now we will find an inner perturbation of $\rho \otimes Ad L$
which is stably conjugate to $\sigma \otimes Ad L$.
Define $u$  by
$u_x = \sum_{z \in \tilde{G}} e_z \xi(z^{-1},x)$.  Then we have
$$(\rho_x \otimes Ad L(x))(u_y) =
\sum_{z \in \tilde{G}} e_{xz} \xi(z^{-1},y) =
\sum_{w \in \tilde{G}} e_w \xi(w^{-1}x,y)$$
where $w = xz$.  This implies that
$$u_x (\rho_x \otimes Ad L(x))(u_y) =
\sum_{w \in \tilde{G}} e_w \xi(w^{-1},x) \xi(w^{-1}x,y)$$
From the cocycle relations for $\xi$, this is
$$\sum_{w \in \tilde{G}} e_w \xi(x,y) \xi(w^{-1},xy)$$ which is equal
to $\xi(x,y)u_{xy}$.  This means that the $u_g$'s form a twisted
unitary cocycle for $\rho \otimes Ad L$, with associated
scalar 2-cocycle $\xi$.

$\xi$ induces the characteristic invariant
$(\lambda,\mu)_{\sigma}-(\lambda,\mu)_{\rho}$.
Therefore $Ad u \circ (\rho \otimes Ad L)$ has the same characteristic
invariant as $\sigma$, and is stably conjugate to $\sigma$.
So from the previous section
$\sigma$ induces the same subfactor as $Ad u \circ (\rho \otimes Ad L)$.

Since $\xi|_{H}=0$ and $\xi_{K}=0$, the restrictions
$u|_{H}$ and $u|_{K}$
are unitary cocycles without twist for $\rho \otimes Ad L$.
$H$ and $K$ have outer actions, so from~\cite{JoT}
there exist $w_H$ and $w_K$
such that for $h$ in $H$, $k$ in $K$ we have
$$u_h = w_H (\rho_h \otimes Ad L(h))(w_H^*)$$ 
and likewise
$$u_k= w_K (\rho_k \otimes Ad L(k))(w_K^*)$$

Using the same argument as in the proof of lemma 4.2,
this implies that $Ad u \circ (\rho \otimes Ad L)$
and $\rho \otimes Ad L$ induce isomorphic subfactors.
From the previous section,
$Ad u \circ (\rho \otimes Ad L)$ and $\rho$ do so as well.

We conclude that $\rho$ and $\sigma$
induce isomorphic subfactors.
\end{proof}

If $\ker(\rho) \cap \ker(\sigma)$ is infinite index in $H \ast K$,
then $\tilde{G}$ is infinite and the above theorem cannot be applied.
In such cases it may be useful to perturb $\rho$ or $\sigma$
by appropriate unitary cocycles.
This will not change the outer conjugacy classes of the representations
or the isomorphism classes of the subfactors, but it can affect
$\ker(\rho)$ and $\ker(\sigma)$, and possibly reduce
$\tilde{G}$ to a more manageable size.

Note that automorphisms of the group algebras $\mathbb{C}[H]$
and $\mathbb{C}[K]$ cannot affect
the subfactor $M^H \subset M \rtimes K$.
In the case of abelian groups, this is just
automorphisms of $H$ and $K$ separately.  In general, however,
all we can say is that outer conjugacy up to separate automorphisms
of the group algebras is also sufficient for subfactor isomorphism.

The original map $\alpha$ which induced the outer conjugacy
had the property that $\alpha(P)=M$.  This property is preserved
at every step of the above argument.  In other words,
outer conjugacy implies not only $M^H \subset M \rtimes K \cong
P^H \subset P \rtimes K$, but
$M^H \subset M \subset M \rtimes K \cong P^H \subset P \subset P \rtimes K$.
We call this isomorphism of towers a triplet isomorphism.

\subsection{Subfactor Triplet Isomorphism Implies Outer Conjugacy}

We will now provide the converse (approximately)
 of the proceeding theorem.
As before, let $H$ and $K$ be finite groups,
$\rho$ and $\sigma$ actions of $H \ast K$ 
on $M$, $P$ respectively with outer restrictions to $H$ and $K$.

\begin{thm}
Let $M^H \subset M \rtimes K$ be isomorphic to
$P^H \subset P \rtimes K$ via $\alpha:P \rtimes K \rightarrow M \rtimes K$,
i.e. $\alpha(P) = M$, $\alpha(P^H) = M^H$.
Then $\rho$ and $\sigma$ are outer conjugate.
\end{thm}
\begin{proof}
Take $\tilde{\sigma}=\alpha \sigma \alpha^{-1}$.  This gives
actions of $H$ and $K$ on $\alpha(P) = M$.
Since $\alpha(P^H)=M^H$, 
$\tilde{\sigma}|_H$ commutes with left and right multiplication
by $M^H$.  Any such linear operator on $\B(L^2(M))$ is
contained in the relative commutant
$ (M^H)' \cap M \rtimes H$; this is equal to
$\mathbb{C}[H]$, where the action of $H$ is implemented by $\rho$.
So $\tilde{\sigma}$
and $\rho$ give the same $H$-action, up to group algebra
automorphism of $H$.

Let $N$ be the fixed-point algebra of $M$ under the action
of $\tilde{\sigma}|_K$.  
Since $\alpha(P)=M$ and $\alpha(P \rtimes K)=M \rtimes K$,
we have $N \subset M \subset M \rtimes K$ isomorphic to the basic
construction on $N \subset M$.
Therefore $N$ differs from $M^H$ by an inner automorphism,
i.e. $N=uM^Ku^*$ for some unitary $u \in M$.  It follows as above that
$\tilde{\sigma}_k=u\rho_ku^*$ for $k \in K$, up to group algebra
automorphism of $K$.

We conclude that up to inner perturbation,
$\tilde{\sigma}$ and $\rho$ agree on $H$ and $K$,
and hence on the free product.
So triplet automorphism of the corresponding subfactors
implies that $\rho$ and $\sigma$
are outer conjugate, up to separate automorphisms of the individual
group algebras.
\end{proof}

\subsection{Exceptions}

We will now discuss our choice of assumptions.

Let $H$ be a finite abelian group acting on a $II_1$ factor, not cyclic,
with
action $\rho$ on the hyperfinite $II_1$ factor $M$.
$\rho_h$ is implemented in
the crossed product $M \rtimes_{\rho} H$ by $\{u_h\}$.
Let
$\{w_h|h \in H\}$ be a
twisted unitary cocycle for $\rho$ with nontrivial scalar
2-cohomology.  Let $\sigma_h= Ad w_h \circ \rho_h$ for $h \in H$,
implemented in the crossed product $M \rtimes_{\sigma} H$ by $\{v_h\}$.

Now we consider the subfactors
$M^H \subset M \rtimes_{\rho} H$ and
$M^H \subset M \rtimes_{\sigma} H$, where the fixed
point algebra is taken relative to the action
of $\rho$ in both cases.  Taking $H=K$,
these are both Bisch-Haagerup subfactors, and the actions are
clearly outer conjugate.

For the first subfactor,
$(M^H)' \cap M \rtimes_{\rho} H$ is spanned by elements $u_h$.
Since $H$ is abelian, this relative commutant is abelian as well.

For the second subfactor,
since $Ad w_h^* v_h = Ad w_h^* \circ \sigma_h =
\rho_h$,
the relative commutant $(M^H)' \cap M \rtimes_{\sigma} H$
is spanned by
elements of the form $w_h^* v_h$, $h \in H$.
Since the twisted cocycle $\{w_h\}$ has nontrivial scalar
2-cohomology, these elements do not commute with
each other, instead obeying various relations of the form
$xy=\lambda yx$, $\lambda \neq 1 \in \mathbb{C}$, $x$, $y$ in
the relative commutant.  So the standard invariant
of this subfactor is not the same as that of the first one,
and the two subfactors are not isomorphic.

This example is extremely degenerate, but it shows that some
additional assumption is needed for outer conjugacy to imply
subfactor isomorphism.  The result of ~\cite{BDG2} implies
that the assumption used in theorem 4.1, of
the quotient $H\ast K \rightarrow H \ast K / N = H \oplus K$
factoring through $G = H \ast K / Int$, is stronger than necessary.
However, this assumption is always valid
in the twisted tensor product examples described above.

With the above assumption on the action, outer conjugacy is equivalent to triplet
subfactor isomorphism.  This is sometimes strictly stronger than subfactor
isomorphism;
in general, an isomorphism $\alpha$
from $P^H \subset P \rtimes K$ to $M^H \subset M \rtimes K$
need not send $P$ to $M$.
It may be the case that an automorphism
of $M \rtimes K$ exists which leaves $M^H$ invariant and
sends $\alpha(P)$ to $M$,
but it is not clear that such an automorphism should always exists.

\subsection{Isomorphisms of Bisch-Haagerup Subfactors}

In~\cite{BNP}, the authors consider subfactors $M^H \subset M \rtimes K$
where $H$ is abelian and $K$ is prime-order cyclic.  They show that
the normalizer of $M^H$ in $M \rtimes K$ is equal to $M$ for
any such subfactors.
Since normalizers are preserved by isomorphism, this means that
subfactor isomorphism implies triplet isomorphism in all such cases.

We may always detect intermediate subfactors of $M^H \subset
M \rtimes K$ by examining all the projections in the first
relative commutant~\cite{Bi}.
Using~\cite{BH} to compute this commutant shows that the condition
$hkh^{-1}k^{-1} \not \in Int $
for $h\neq 1 \in H$, $k \neq 1 \in K$ is sufficient to rule out additional
intermediate subfactors of the correct index,
when combined with the condition in theorem 4.1
on the quotient of $H\ast K$ by $N_1$.  This again means that
subfactor isomorphism implies triplet isomorphism.

Summarizing the results of this section,
we consider actions of $H \ast K$ on a $II_1$ factor, with outer
restrictions to $H$ and $K$, such that
the inner subgroup of this free product is contained in the first
commutator subgroup of $H \ast K$.  For such actions, outer
conjugacy of the action is equivalent to triplet subfactor
isomorphism $M^H \subset M \subset M \rtimes K \cong
P^H \subset P \subset P \rtimes K$.  Futher assuming that
$M$ is the only intermediate subfactor of its index
in $M^H \subset M \rtimes K$, outer conjugacy of the action
is equivalent to subfactor isomorphism.  $H$ abelian
and $K$ prime-order cyclic gives the same result.

\section{Applications to Hadamard Subfactors}

\subsection{Introduction}

We now take the outer actions of $H$ and $K$ on the hyperfinite
$II_1$ factor $P$ to come from the twisted tensor
product of two group Hadamard matrices, as described in section
3.5.

In this case, the condition in theorem 4.1 on the action
of $H \ast K$ is always true.  From section 3.6,
the inner subgroup of $H \ast K$
is contained in the first commutator subgroup $N_1$.
Therefore the quotient map $H \ast K \rightarrow H \ast K / N_1 =
H \oplus K$ factors through $H \ast K / Int$, and may be defined on
$H \ast K / Int = G$.  This quotient map $(p_H, p_K)$ has the
properties in the assumption of the theorem.
It follows that outer conjugacy of actions implies subfactor
isomorphism in the Hadamard case.

Furthermore, $H$ and $K$ must be
abelian, so the result of~\cite{BNP} will frequently apply;
if $H$ or $K$ is cyclic, then subfactor isomorphism
implies triplet isomorphism, and is therefore
equivalent to outer conjugacy.

Let $G=H \ast K / Int P$ be finite. Let $\tilde{G} \subset Aut P$ be
a finite group with
$\tilde{G} / Int P = G$ and $\tilde{G} \cap Int P = S$.
From~\cite{JoT}, such a finite $\tilde{G}$ always exists.
We will take $\tilde{G}$ to act on $P$ via the representation $\rho$.
The inner subgroup $S$ is in general nontrivial.
As above, we know that $S$
must be contained in the first commutator subgroup $\tilde{N}$
of $\tilde{G}$.  Since $\tilde{G}$ has a compatible
action, from theorem 2.3
each element of $S$ may be implemented by
some $u_s \in B_0 = \Delta_m = l^{\infty}(H)$,
where the tower of $B_i$'s is as in section 3.

We now compute the characteristic invariant of $\tilde{G}$.
$B_0=l^{\infty}(H)$, so we may consider $u_s$ to be a vector
with components labeled by elements of $H$.
Since the $u_s$'s are
only determined up to scalars, we may require
$(u_s)_1=1$.  It follows that if
$u_s u_s'=\mu u_{ss'}$ for some scalar $\mu$, then $\mu=1$.
Therefore $\mu$ is trivial for these actions.

Each element $g$ of $\tilde{G}$ may be written as
$g = hkn$, $h \in H$, $k \in K$, $n \in \tilde{N} = N_1 / \ker(\rho)$.
Since for the Hadamard action
$K$ and $\tilde{N}$ act trivially on $B_0$, $\rho_{kn}(u_s)=u_s$
for $k$, $n$ as above and $s \in S$.
This means we have
$\lambda(kn,s)=1$ and
$\lambda(hkn,s)=\lambda(h,s)$ for $k$, $n$, $s$ as above
and $h \in H$.  Therefore the characteristic
invariant is determined by $\lambda|_{H \times S}$.

From the definition of $\lambda$ 
we have $\rho_h(u_{h^{-1}sh})=\lambda(h,s)u_s$.  We may
determine this scalar by examining the first component.
Since $(\lambda(h,s)u_s)_1=1$ from our choice of $u_s$,
and $\rho_h$ acts via the left regular representation
on the minimal projections of $B_0$,
we may compute $\lambda(h,s)=\overline{\rho_h(u_{h^{-1}sh})_1}
=\overline{(u_{h^{-1}sh})_{h^{-1}} }$.  So
the coordinates of the $u_s$'s determine the characteristic
invariant, and vice versa.

It follows that the existence of any nontrivial
inner subgroup implies a nonzero characteristic invariant,
although the induced element of $H^3(G)$ may sometimes
still be a coboundary.  In addition, if
$\rho$ and $\sigma$ give actions of the same group
$\tilde{G}$ on $P$, with the same inner subgroup $S \subset \tilde{G}$,
and the inner elements $\{u_s\}$ have the same coordinates, then their
characteristic invariants are the same.  In such a case
the 3-cocycle obstructions are the same, the actions
are outer conjugate, and the subfactors are isomorphic.

Applying~\cite{BH} to find the principal graph
for these Hadamard group actions is in some ways
easier than in the general case.  The local freeness
condition of~\cite{BH}
($hgk=x$ for $h \in H, k \in K, g \in G$ only if $h=k=1$)
will always apply, so the odd vertices of the principal
graph correspond to the $H-K$ double cosets $\{HnK|n \in N\}$.
Even vertices in the principal graph correspond to $H-H$ double
cosets $HgH$.  We find the edges of the graph by
decomposing $HnKH$ into such double cosets.  From the above
description of $G$, if $k \neq k'$ then
$HnkH$ and $Hnk'H$ are disjoint, so there is always
one such double coset for each element of $K$.
Finding the number of single $H$-cosets in
each $HnkH$ (taking advantage of the relatively simple
multiplication table of $G$)
then allows us to complete the principal graph.
The dual principal graph is computed similarly.

\subsection{Fourier-4}

Let $H=K=\Z_2$.  In this case $H_1$ and $H_2$ are
both $\frac{1}{\sqrt{2}}\left(\begin{array}{cc}1&1\\1&-1\end{array}\right)$,
and their tensor product is the unique real 4 by 4 Hadamard
matrix
$$\frac{1}{2}\left(\begin{array}{cccc}
1&1&1&1\\
1&-1&1&-1\\
1&1&-1&-1\\
1&-1&-1&1\end{array}\right)$$
We may write the twist
as $(\alpha, \beta, \gamma, \delta) \in
l^{\infty}(\Z_2 \oplus \Z_2)$.  Putting the twist into standard
form sends all the parameters to 1 except $\delta$.
Applying a twist of $T=(1,1,1,\delta)$ gives the twisted
tensor product 
$$H=\frac{1}{2}\left(\begin{array}{cccc}
1&1&1&1\\
1&-1&1&-1\\
1&\delta&-1&-\delta\\
1&-\delta&-1&\delta\end{array}\right)$$

All size-4 complex Hadamard matrices are contained in a single
1-parameter family~\cite{Ho}.
As $\delta$ takes on values in the torus,
$H$ varies over this entire family.  In other words,
all 4 by 4 Hadamard matrices are twisted tensor products.
                        
We know that the group $H \ast K / Int$ is equal to $HKN$ from
section 3.6.  $H$ and $K$ are each generated by a single order-2
element, respectively $h$ and $k$.

$N$ has a single generator $n = hkhk$.
From section 3.6, this compatible
automorphism is induced by $Ad b(n)$, where 
$$b(n) = \tilde{T} \rho_h(\tilde{T}^*) \rho_k(\tilde{T}^*) \rho_{hk}(\tilde{T})
\in B_0' \cap N_1$$
For convenience we write elements of $N \subset B_0' \cap B_1$ in
array form, since this abelian algebra has one minimal projection
for each pair $(h,k)$, $h \in H$, $k \in K$.  These are not
matrices--multiplication is still pointwise.
We will consider
rows to be labeled by elements of $H$, columns by elements of $K$.
So $$b(n) =
\left[\begin{array}{cc}\delta&\overline{\delta}\\
\overline{\delta}&\delta\end{array}\right]
$$

Multiplying columns by scalars has no affect on the induced compatible
automorphism,
since $Z(B_1)=l^{\infty}(K)$ has trivial adjoint action on $B_1$.
So the action of $n$ is also induced by
$$\left[\begin{array}{cc}1&1\\
\overline{\delta^2}&\delta^2\end{array}\right]
$$
Taking the $l$th power gives $n^l$ induced by
$$\left[\begin{array}{cc}1&1\\
\overline{\delta^{2l}}&\delta^{2l}\end{array}\right]
$$

This element of $B_0' \cap B_1$ will induce an inner automorphism
when it is constant along each row.  This occurs when
$\overline{\delta^{2l}} = \delta^{2l}$, i.e. when $\delta^{2l}=\pm 1$.

Suppose that $\delta$ is a rational rotation.
Then let $l$ be the smallest natural number such that
$\delta^{4l}=1$.  This means that $\delta^{2l} = \pm 1$,
and $N = \Z_l$, 
generated by $n=hkhk$.  $hnh=knk=n^{-1}$,
implying that $G = H \ast K / Int = HKN / Int$ is the dihedral
group $D_{2l}$.
Generators are $s=h$ and $t=hk$, with $t^{2k}=s^2=stst=1$.

If $\delta$ is an irrational rotation, this group is $D_{\infty}$.

Applying the methods of~\cite{BH} then gives principal graphs for
all 4 by 4 Hadamard matrices, respectively $D_{2l+1}^{(1)}$ and $D_{\infty}$.

The above classification cannot be complete, as may be seen by
examining the case $\delta=i$.  It may be shown that
the resulting matrix is Hadamard equivalent to $H_{\Z_4}$,
while if $\delta=1$ then the matrix is $H_{\Z_2 \oplus \Z_2}$.
These subfactors have the same principal graph,
but they are not isomorphic.

As described above, we may classify the index-4 case
up to subfactor isomorphism
by classifying the action of the free product
up to outer conjugacy, since the necessary conditions
are satisfied.

Now we consider cohomology.  If $\delta$ is rational,
we take as before
$\delta^{4l}=1$, fixing the principal graph.
We now consider the two cases $\delta^{2l}=1$
and $\delta^{2l}=-1$.

If $\delta^{2l}=1$, then $(hkhk)^n$ is the identity, and we have
a true outer action of $H \ast K = D_{2l}$.  Any two outer actions
of a finite group are conjugate~\cite{JoT},
so for any choice of $\delta$ with $\delta^{2l}=1$, the corresponding
Hadamard subfactors are isomorphic.

Now suppose $\delta^{2l}=-1$.  In this case $(hkhk)^n$ is a nontrivial
inner $u = (1, -1) \in B_0=\Delta_m \otimes 1$, with $u^2 =1$.  This allows
us to extend the representation of $H \ast K=D_{2l}$ in $Out P$ to an action
of $D_{4l}$ on $P$ with inner subgroup $\{1,u\}=\Z_2$.
$u$ does not depend on the particular choice of root, so the characteristic
invariant and associated 3-cocycle of $G=D_{2n}$ are the same
for any such $\delta$.  It follows that all such subfactors are again
isomorphic.

This means that there are at most 2 nonisomorphic Hadamard subfactors
with each graph $D_{2l+1}^{(1)}$ at index 4.   In~\cite{IK} the authors
find that there are $n-2$ subfactors with principal graph
$D_n^{(1)}$ for any $n$.
 This means
that many of these subfactors cannot be constructed from Hadamard
commuting squares.

It remains to show that the case $\delta^{2l} = \pm 1 $ give distinct
subfactors.  For $l=1$ the subfactors are depth 2,
and may be identified by direct computation of the
intermediate subfactor lattice (using~\cite{JoP}) as
$N \subset N \rtimes \Z_2^2$ and
$N \subset N \rtimes \Z_4$ for $\delta^{2l}= 1, -1$ respectively.

To prove that the cases $\delta^{2l}= \pm 1$
are distinct for larger $l$, we must show that the associated
3-cocycles are different.

We consider the cyclic subgroup of $D_{4l}$ generated by $a=hk$ in the
case $\delta^{2l}=-1$.
$a$ has order $2l$ in $Out P$, with $a^l=Ad u$.  Now, $\lambda(a,a^l)=-1$,
since $hk(u)=-u$.  This is a nontrivial characteristic invariant,
but cyclic groups have trivial 2-cohomology, so it does not
come from a 2-cocycle on the subgroup.  Therefore from the exact
sequence of~\cite{JoT} the kernel of this subgroup
has nontrivial associated 3-cocycle,
implying the kernel of the full group does as well. Since
the cocycle is trivial in the case $\delta=1$, the two corresponding
Hadamard subfactors are not isomorphic.

This completes the classification of the index-4
Hadamard subfactors.

\subsection{Hadamard 16-7}

There are five real 16 by 16 Hadamard matrices.  Numerical
computations give the dimensions of the first relative
commutants as 16, 7, 4, 3, 3.  The first one is depth 2,
and the last 3 have excessively sparse intermediate subfactor
lattices to be Bisch-Haagerup.  Hadamard 16-7, however,
may be obtained as the twisted tensor product of the unique
real $4 \times 4$ Hadamard matrix with itself, using the twist
$T=(1,1,1,...,1,-1)$.

In this case we have $H=K=\Z_2 \oplus \Z_2$.  The first commutator
subgroup $N$ is induced by unitaries in $B_0' \cap B_1$.
However, all coefficients are real (i.e. $\pm 1$) so every element
in $N$ must be order 2.

We will show that $N$ is in fact $\Z_2^4$, and that $H\ast K/Int$
is therefore a non-abelian group of order 256.  We will be able
to compute the principal graph for the subfactor as well.

$N$ is generated by the nine elements of the form $hkhk$, for
$h$ and $k$ nontrivial elements of $H$, $K$ respectively.
Each such element $n$ is induced by
$b(n) \in l^{\infty}(H) \otimes l^{\infty}(k)$.
We write these unitaries in array form, as in the previous section.
We take $H=\{1,w,x,wx\}$ and $K=\{1,y,z,yz\}$,
with the rows and columns numbered in that order.  All coefficients
will be $\pm 1$, and we label them by sign.  Again rows are indexed by
 $H$, columns by $K$.

$$b(1)=
\left[
\begin{array}{cccc}
+&+&+&+\\
+&+&+&+\\
+&+&+&+\\
+&+&+&+\end{array}
\right]
\quad
b(wywy)=
\left[
\begin{array}{cccc}
+&+&+&+\\
+&+&+&+\\
+&+&-&-\\
+&+&-&-\end{array}
\right]
$$
$$b(wzwz)=
\left[
\begin{array}{cccc}
+&+&+&+\\
+&+&+&+\\
+&-&+&-\\
+&-&+&-\end{array}
\right]
\quad
b(wyzwyz)=
\left[
\begin{array}{cccc}
+&+&+&+\\
+&+&+&+\\
-&+&+&-\\
-&+&+&-\end{array}
\right]
$$
$$b(xyxy)=
\left[
\begin{array}{cccc}
+&+&+&+\\
+&+&-&-\\
+&+&+&+\\
+&+&-&-\end{array}
\right]
\quad
b(xzxz)=
\left[
\begin{array}{cccc}
+&+&+&+\\
+&-&+&-\\
+&+&+&+\\
+&-&+&-\end{array}
\right]
$$
$$b(xyzxyz)=
\left[
\begin{array}{cccc}
+&+&+&+\\
-&+&+&-\\
+&+&+&+\\
-&+&+&-\end{array}
\right]
\quad
b(wxywxy)=
\left[
\begin{array}{cccc}
+&+&-&-\\
+&+&+&+\\
+&+&+&+\\
+&+&-&-\end{array}
\right]
$$
$$b(wxzwxz)=
\left[
\begin{array}{cccc}
+&-&+&-\\
+&+&+&+\\
+&+&+&+\\
+&-&+&-\end{array}
\right]
\quad
b(wxyzwxyz)=
\left[
\begin{array}{cccc}
-&+&+&-\\
+&+&+&+\\
+&+&+&+\\
-&+&+&-\end{array}
\right]
$$
Clearly every element of $N$ is order 2.

Multiplying a row of some $b(n)$ by -1 corresponds to an inner automorphism,
while multiplying a column by -1 is trivial.  We can rewrite the
above generators in standard form, and ignore the first row
and column (since these are always 1's).  This gives us each
generator as having four minus signs, with an even number in each
row and column.  This parity condition gives five relations, and
any product of generators will still have
this property, so these elements span a subspace of at most dimension
4 in $\Z_2^9$.  In fact \{$wywy$, $wzwz$, $xyxy$, $xzxz$\} is a minimal
generating set for $N$,
with the other elements obeying the relations

$(wywy)(wzwz) = wyzwyz$, $(xyxy)(xzxz) = xyzxyz$,

$(wywy)(xyxy)=wxywxy$, $(wzwz)(xzxz)=wxzwxz$,

and $(wywy)(wzwz)(xyxy)(xzxz)=wxyzwxyz$.

All of
these relations are valid in $Out P$, but some require nontrivial
inner adjustment.

Next we note that (in $Out$) $N$ is central.  
We compute
$$yhkhky = (yhyh)(hykhyk)$$
for $h \in H$, $k \in K$.
This will be
$(hyhy)(hyhy)(hkhk) = hkhk$ for any $h$ and $k$.
So $hkhk$ commutes with $y$.  It may be simililarly shown that $hkhk$
commutes
with every other element of $H$ and $K$, and hence
with their entire free product.

We know that $G = H \ast K / Int$ for Hadamard subfactors will
be of the form $HKN$.  $N$ is order 16, so
$|G|=|H||K||N|=4\cdot 4\cdot 16 = 256$.  We have enough data
to determine the multiplication table for $G$.  Let
$h_1, h_2 \in H$, $k_1, k_2 \in K$, and $n_1, n_2, h_2k_1h_2k_1 \in N$.
$$(h_1k_1n_1)(h_2k_2n_2)$$
$$= (h_1k_1)(h_2k_2)(n_1n_2)$$
$$=(h_2h_2)(h_2k_1h_2k_1)(k_1k_2)(n_1n_2)$$
$$=(h_1h_2)(k_1k_2)(n_1n_2)(h_2k_1h_2k_1)$$
The above relations will always allow us to express
$h_2k_1h_2k_1$ in terms of our four generators of $N$, providing
the multiplication table for the group.  We can identify
this group as number 8935 of its order in the MAGMA small-group catalog.

We may now use the methods of~\cite{BH} to find the principal graph.
Let $h \in H$, $k \in K$.

For $g \in G$, $hgk=g$ only if $h=k=1$, so the group is locally
free.  This
means that the odd vertices of the graph correspond to the
16 elements of $N$, i.e. to the $H-K$ double cosets $HnK$.

Even vertices are divided into classes according to the double
coset structure $HGH$.  A double coset $HgH$ will contain
4 elements if $g$ is in $HN$ (and hence commutes with $H$).
If $g=kn$ for $k \neq 1$, then $HgH$ contains the 16 distinct
elements of the form $hk h'kh'k n$ for $h, h' \in H$.
A 4-element double coset corresponds to a cluster of
four even vertices, each
representing an irreducible bimodule of
H-dimension 1.  A 16-element double coset corresponds to
a single bimodule of size $4^2=16$.  We have 12 single vertices
and 16 clusters, for a total of $64+12=76$ even vertices.

An odd vertex $HnK$ is connected to an even vertex
$HgH$ once for each time that the bimodule $HgH$ occurs
in the product $HnKH$.  Every even vertex in a cluster
is connected to the same odd vertices.

Since $n \in N$ is central, $HnKH = HKHn$.
$HKH$ decomposes as $HyH \cup HzH \cup HyzH \cup H1H$.  
So for any $n \in N$, $HnK$ is connected to the vertices
$HyHn=HynH$, $HznH$, $HyznH$ and the four-vertex cluster $Hn$.
The cluster $Hn$ connects only to $HnK$, and the vertex
$HknH$ connects to the four vertices $H(hkhk)nH$, $h \in H$.
This fully describes the principal graph.

Some computer simulations suggest that the 3-cocycle associated with
this subfactor is trivial, but this is not certain.

\subsection{Fourier-6}

Let $H=\Z_2$, $K=\Z_3$.  Both of these groups are prime-order cyclic, so from
section 4.7 and~\cite{BNP} kernel conjugacy is equivalent to subfactor
isomorphism, up to automorphism of the two small groups.
$H$ has one nontrivial automorphism, so this will be relevant.
$H$ is generated by $h$, $K$ by $k$.
We construct twisted tensor product of the depth-2 Hadamard
matrices corresponding to $H$ and $K$.

The first commutator subgroup for any twisted tensor product
will be an abelian group with two generators,
namely $x=hkhk^2$ and $y=hk^2hk$.  Each one of these generators
may be represented as a unitary in $\mathbb{C}[H \oplus K]$ as
above.

Let the twist be $(1,1,1,1,\chi,\xi)$, in standard form.
In this case we compute from section 3.6
$$b(x)=\left[\begin{array}{ccc}
\xi&\overline{\chi}&\overline{\xi}\chi\\
\overline{\xi}&\chi&\xi\overline{\chi}
\end{array}\right],
b(y)=\left[\begin{array}{ccc}
\chi&\overline{\chi}\xi&\overline{\xi}\\
\overline{\chi}&\chi\overline{\xi}&\xi
\end{array}\right]$$

Multiplying a column by a scalar is trivial, so $x$
and $y$ are respectively induced by
$$\left[\begin{array}{ccc}
1&1&1\\
\overline{\xi}^2&\chi^2&\xi^2\overline{\chi}^2
\end{array}\right],
\left[\begin{array}{ccc}
1&1&1\\
\overline{\chi}^2&\chi^2\overline{\xi}^2&\xi^2
\end{array}\right]$$

The first thing we want is the principal graph.  To find
this we may freely perturb $x$ and $y$ by inners.
Specifically we multiply the second row of
the above elements of $l^{\infty}(H) \otimes l^{\infty}(K)$
by $\xi^2$, $\chi^2$
respectively.  This gives the first commutator subgroup $N$ of
$G\ast H/Int=K$ as the subgroup of $(S^1)^2$
generated by
$(\chi^2\xi^2, \overline{\chi}^2\xi^4)$ and
$(\chi^4\overline{\xi}^2, \chi^2\xi^2)$.

We may describe these elements in additive notation, taking $\chi=e^{2\pi a}$,
$\xi=e^{2\pi b}$.  Then in $\mathbb{R}/\mathbb{Z}$,
these generators are $(s,s+t)$ and $(s+t,t)$ in for $s=4a-2b$, $t=4b-2a$.
The group will be finite if and only if $s$ and $t$ are both
rational, or equivalently if $a$ and $b$ are.

In this finite case, $N$ will be some finite subgroup of $\Z \oplus \Z$,
which may be directly computed without difficulty for any particular
$\chi$ and $\xi$.  Computations with the generators
give $hxh=x^{-1}$, $hyh=y^{-1}$, $kxk^2=x^{-1}y$, $kyk^2=x^{-1}$.
These relations provide
a complete multiplication table for the group $G=HKN$,
and so we can use~\cite{BH} to find the principal graph.

Since we have local freeness, odd vertices will be indexed
by double cosets $HnK$. Each double coset $HgH$ containing
$p$ single cosets $gH$ will
correspond to a cluster of $|H|/p$ even vertices, all connecting
to the same odd vertices.  Connections on the graph are
determined by breaking up $HnKH$ as a sum of double cosets
$HgH$; the vertex $HnKH$ connects to every vertex or cluster
represented in this sum, with multiplicities determined by the number
of times each $HgH$ appears.

To compare two different twists, we pick some
$l$ sufficiently
large so that all components of both twists are $l$th roots
of unity.  We then have two actions of $\tilde{G}$, with subfactor
isomorphism being equivalent to conjugacy of $G$-kernels (up
to the automorphism of $K=\Z_3$).  We can compute
the characteristic invariant, which
will sometimes allow us to assert that certain subfactors with
a given principal graph are isomorphic.

Since $(hk)^6$ and $(hk^2)^6$ act trivially for any choice
of twist (i.e., not via an inner automorphism), restriction
to cyclic subgroups does not usually provide nontrivial 3-cohomology.
We will not in general be able to assert that two
of these subfactors with the same $G$ are different--even
if the characteristic invariants are different, the 3-cocycles
might be the same.

We now give principal graphs for a few simple twists of
$H_{\Z_2} \otimes H_{\Z_3}$.

$H=\Z_3$, $K=\Z_2$, $T=(1,1,1,1,1,-1)$:  $G=\Z_6$ and the
subfactor is depth 2, so the associated
3-cocycle is trivial.  However, the characteristic invariant is
nontrivial.  The constraints on the action
of $\tilde{G}$ could imply that every nontrivial characteristic
invariant induces a nontrival 3-cocycle; this example shows
that this is not the case.

$H=\Z_2$, $K=\Z_3$, $T=(1,1,1,1,1,e^{2\pi i/3})$:
We put $x$ and $y$ in standard form as elements of $(S^1)^2$ to find
$H \ast K / Int$. 
Then 
$x=(e^{\frac{2}{3}2 \pi i},e^{\frac{1}{3}2 \pi i/3}),
y=(e^{\frac{1}{3}2 \pi i},e^{\frac{2}{3}2 \pi i/3})$.
We conclude that $x^2=y$ in $Out$, so $N=\Z_3$.  $G$ is a non-abelian
group of order 18.

Principal graph------------Dual principal graph

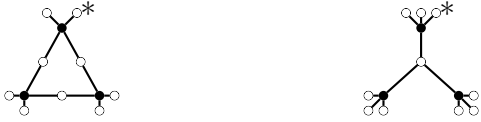
\begin{pspicture}(2,1.6)
\psset{dotstyle=*}
\dotnode(0.3,0.3){A1}
\dotnode(1.3,0.3){A2}
\dotnode(0.8,1.2){A3}

\psset{dotstyle=o}
\dotnode(0.3,0.1){B11}
\dotnode(0.1,0.3){B12}
\dotnode(1.5,0.3){B21}
\dotnode(1.3,0.1){B22}
\dotnode(0.6,1.4){B31}
\dotnode(1.0,1.4){B32}
\dotnode(0.8,0.3){C3}
\dotnode(1.05,0.75){C1}
\dotnode(0.55,0.75){C2}

\rput(1.15,1.4){*}

\ncline{B11}{A1}
\ncline{B12}{A1}
\ncline{B21}{A2}
\ncline{B22}{A2}
\ncline{B31}{A3}
\ncline{B32}{A3}

\ncline{C3}{A1}
\ncline{C3}{A2}
\ncline{C1}{A2}
\ncline{C1}{A3}
\ncline{C2}{A1}
\ncline{C2}{A3}
\end{pspicture}
\begin{pspicture}(2.5,1)
\end{pspicture}
\begin{pspicture}(2,1.6)

\psset{dotstyle=*}
\dotnode(0.3,0.3){A1}
\dotnode(1.3,0.3){A2}
\dotnode(0.8,1.2){A3}

\psset{dotstyle=o}
\dotnode(0.3,0.1){B11}
\dotnode(0.1,0.3){B12}
\dotnode(0.1,0.1){B13}
\dotnode(1.5,0.3){B21}
\dotnode(1.3,0.1){B22}
\dotnode(1.5,0.1){B23}
\dotnode(0.6,1.4){B31}
\dotnode(1.0,1.4){B32}
\dotnode(0.8,1.4){B33}
\dotnode(0.8,0.75){C}

\rput(1.15,1.4){*}

\ncline{B11}{A1}
\ncline{B12}{A1}
\ncline{B13}{A1}
\ncline{B21}{A2}
\ncline{B22}{A2}
\ncline{B23}{A2}
\ncline{B31}{A3}
\ncline{B32}{A3}
\ncline{B33}{A3}

\ncline{C}{A1}
\ncline{C}{A2}
\ncline{C}{A3}
\end{pspicture}

$H=\Z_2$, $K=\Z_3$, $T=(1,1,1,1,1,i)$: $|G|=24$.
Here we find $N=\Z_2 \oplus \Z_2$, where $x$ and $y$ are the two
generators.

Principal graph------------Dual principal graph

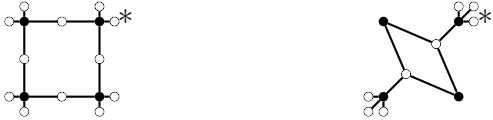
\begin{pspicture}(2,1.6)
\psset{dotstyle=*}
\dotnode(0.3,0.3){A1}
\dotnode(0.3,1.3){A2}
\dotnode(1.3,1.3){A3}
\dotnode(1.3,0.3){A4}

\psset{dotstyle=o}
\dotnode(0.1,0.3){B11}
\dotnode(0.3,0.1){B12}
\dotnode(0.1,1.3){B21}
\dotnode(0.3,1.5){B22}
\dotnode(1.5,1.3){B31}
\dotnode(1.3,1.5){B32}
\dotnode(1.5,0.3){B41}
\dotnode(1.3,0.1){B42}

\dotnode(0.3,0.8){C1}
\dotnode(0.8,1.3){C2}
\dotnode(1.3,0.8){C3}
\dotnode(0.8,0.3){C4}

\rput(1.65,1.3){*}

\ncline{B11}{A1}
\ncline{B12}{A1}
\ncline{B21}{A2}
\ncline{B22}{A2}
\ncline{B31}{A3}
\ncline{B32}{A3}
\ncline{B41}{A4}
\ncline{B42}{A4}

\ncline{C1}{A1}
\ncline{C1}{A2}
\ncline{C2}{A2}
\ncline{C2}{A3}
\ncline{C3}{A3}
\ncline{C3}{A4}
\ncline{C4}{A4}
\ncline{C4}{A1}
\end{pspicture}
\begin{pspicture}(2.5,1)
\end{pspicture}
\begin{pspicture}(1.6,1.6)
\psset{dotstyle=*}
\dotnode(0.3,0.3){A1}
\dotnode(0.3,1.3){A2}
\dotnode(1.3,1.3){A3}
\dotnode(1.3,0.3){A4}

\psset{dotstyle=o}
\dotnode(0.1,0.3){B11}
\dotnode(0.3,0.1){B12}
\dotnode(0.1,0.1){B13}
\dotnode(1.5,1.3){B31}
\dotnode(1.3,1.5){B32}
\dotnode(1.5,1.5){B33}

\dotnode(0.6,0.6){C3}
\dotnode(1,1){C1}

\rput(1.65,1.3){*}

\ncline{B11}{A1}
\ncline{B12}{A1}
\ncline{B13}{A1}
\ncline{B31}{A3}
\ncline{B32}{A3}
\ncline{B33}{A3}

\ncline{C1}{A2}
\ncline{C1}{A3}
\ncline{C1}{A4}
\ncline{C3}{A1}
\ncline{C3}{A2}
\ncline{C3}{A4}

\end{pspicture}

Now let $\xi$ be a primitive 15th root of unity, and consider
$T=(1,1,1,1,1,\xi)$.  We wish to consider cohomology in this case,
so we do not perturb the generators by inner automorphisms (multiplying
a row by a scalar).  However multiplying columns by scalars is still
trivial.

We then have
$b(x)=\left[\begin{array}{ccc}1&1&1\\
                \xi^{-2}&1&\xi^2\end{array}\right]$,
$b(y)=\left[\begin{array}{ccc}1&1&1\\
                1&\xi^{-2}&\xi^2\end{array}\right]$.
We have $x^{15}=y^{15}=1$, but there is an additional relation:
$b(x^5y^5)=\left[\begin{array}{ccc}1&1&1\\
\xi^{-10}&\xi^{-10}&\xi^{20}\end{array}\right]$.  Since
$\xi^{20}=\xi^{-10}$, this element of $l^{\infty}(H) \otimes l^{\infty}(K)$
induces the inner
automorphism $Ad u$, $u=(1, \xi^{-10}) \in \Delta_m$.
It follows that $N = \Z_5 \oplus \Z_{15}$, with generators
$x+y$ and $x$, and $|G|=450$.  The principal graph may be computed
using the same methods as above.

In this case $\tilde{G}$ is an order-1350 group with inner subgroup
$\Z_3$, since $u^3$ is the identity.
The characteristic invariant
is completely determined by $u$, as discussed above.  If we let
$\xi=e^{\frac{a}{15}2\pi i}$, then for $\xi$ to be a primitive
15th root of unity we must have $a \in \{1,2,4,7,8,11,13,14\}$.
Choosing $a \in \{1,4,7,13\}$ gives the same value
for $\xi^{-10}$.  It follows that the corresponding four group
actions have the same characteristic invariant, and therefore
the same 3-cocycle.  This means that
the subfactors are isomorphic.  Likewise choosing $a$ from
$\{2,8,11,14\}$ gives isomorphic subfactors.  Applying
the automorphism of $K$ sending $k$ to  $k^2$ does not
change the characteristic invariant in either case.

These two kinds of roots give different characteristic invariants,
but unlike the index-four case it is not possible to detect
3-cohomology on cyclic subgroups.  Some appropriate abelian
subgroups might allow us to detect 3-cohomology, but for
now it is not clear if these two types of subfactor are isomorphic.

\subsection{Other Examples}

We will now consider the case $H=K=\Z_3$, with twist
$(1,1....,1,\xi)$ for $\xi$ a cube root of unity.  This example
is of interest because the cohomology is more tractible than
in the index-6 case.

First we need to compute $N \subset Out$.  Let $H$ and $K$
have generators $h$,$k$.  Then there are four generators
of $N$, namely
$w=hkh^2k^2$,
$x=hk^2h^2k$,
$y=h^2khk^2$,
$z=h^2k^2hk$.  As usual we may freely multiply columns by scalars
without changing the induced automorphism.

$$b(w)=
\left[\begin{array}{ccc}
\xi&1&\xi^2\\
1&1&1\\
\xi^2&1&\xi
\end{array}\right]$$
so $w$ is induced by
$$
\left[\begin{array}{ccc}
1&1&1\\
\xi^2&1&\xi\\
\xi&1&\xi^2
\end{array}\right]$$
Likewise,
$$
b(x)=
\left[\begin{array}{ccc}
1&\xi&\xi^2\\
1&1&1\\
1&\xi^2&\xi
\end{array}\right]$$
so $x$ is induced by
$$
\left[\begin{array}{ccc}
1&1&1\\
1&\xi^2&\xi\\
1&\xi&\xi^2
\end{array}\right]
$$
and $y$ and $z$ are induced by respectively
$$
\left[\begin{array}{ccc}
1&1&1\\
\xi&1&\xi^2\\
\xi^2&1&\xi\end{array}\right],\quad
\left[\begin{array}{ccc}
1&1&1\\
1&\xi&\xi^2\\
1&\xi^2&\xi\end{array}\right],
$$

We see then that $w^3=x^3=y^3=z^3=1$, $w=y^2, x=z^2$.
Finally $y=Aduz^2$, where $u=(1,\xi,\xi^2) \in B_0$ induces
a compatible inner automorphism.  So $N=Z_3$, and $G=H \ast K /Out$
is order 81.  The extension to $\tilde{G}$ is order 243,
since the inner subgroup $S=\{1,Adu,Adu^*\}$ is order 3.  $s=Ad u$
is a generator of $S$.

Now we consider the cyclic subgroup generated by $hk$.
$(hk)^3 = hkhkhk=(hkh^2k^2)(kh^2k^2h)(h^2k^2hk) = wy^{-1}z=u^*=s^2$.
$u^*$ is order 3, so $hk$ generates an order-9 cyclic subgroup
in $\tilde{G}$, with inner subgroup $Z_3$.  $hk(u^*)=h(u^*)=(\xi,1,\xi^2)=\xi u^*$,
so the associated characteristic invariant coefficient
$\lambda(hk,s^2)$
is equal to $\xi$.

Since $H^2(\Z_9)$ is trivial, the exact sequence of~\cite{JoT} implies
that the map from $\Lambda_{\Z_9, \Z_3}$ to $H^3(\Z_3)$ is injective.
This means that the restriction of $G$ to the cyclic subgroup generated
by $hk$ has nontrivial 3-cocycle, so $G$ itself does as well.
The given representation of $G$ in $Out M$ does not lift to $Aut$.

The two possible values of $\xi$ give different
characteristic invariants, with different
restrictions to the above cyclic subgroup.  This
might suggest that the associated subfactors are different,
and indeed the actions of $H \ast K$ given by the two
twists are not outer conjugate.

However, performing the automorphism of $H$
given by $h \rightarrow h^2$ switches the second and
third basic projections of $B_0$.  This sends $u$ to $u^*$.
In fact, this tranformation 
changes the entire first commutator subgroup
in $B_0' \cap B_1$ for
$\xi=e^{2\pi i/3}$ to that of
$\xi=e^{4\pi i/3}$,
and therefore sends the first characteristic invariant
to the second.  Such a group automorphism corresponds
to an automorphism of $\mathbb{C}[H]$, and does not change
the subfactor.  The result of $~\cite{BNP}$ applies here,
so equivalence of characteristic invariants
gives subfactor isomorphism.
It follows that the two values of $\xi$
in fact give isomorphic subfactors.

\subsection{Infinite-Depth Hadamard Subfactors}

We now obtain a new infinite-depth Hadamard
subfactor.  Let $H=\Z_2$, $K=\Z_3$.  If the two
twist parameters are
mutually irrational, then $N$ is equal to $\Z_2^2$, and the group
is $G_{2,3,6}$ of~\cite{BH}. The corresponding
Hadamard subfactor is of infinite depth.   The principal
graph for this subfactor is given in~\cite{BH}.

The same construction gives a family of infinite-depth
Hadamard subfactors.
For any two finite
abelian groups $H$, $K$, we may take a generic twist with
all entries mutually irrational.  We will likewise obtain
$N=\Z^{(|H|-1)(|K|-1)}$, and find an infinite-depth Hadamard
subfactor of index $|H||K|$.

For all of these subfactors, $G=H\ast K/Int$ has a finite-index
abelian subgroup, namely $N$.  Therefore $G$ is always amenable, and
from~\cite{BH} these
these subfactors are amenable as well.  In fact $G$ displays polynomial
growth in its generators, so the entropy conditions of~\cite{BH} apply,
and the subfactors are strongly amenable.  No examples of nonamenable
Hadamard subfactors are currently known.

\subsection{Non-Commutative Groups}

Let $G$ be any finite group, not necessarily commutative, $|G|=n$.
Just as in the abelian case, we may construct the canonical
commuting square
$$\commsq{\mathbb{C}}{l^{\infty}(G)}{\mathbb{C}[G]}{M_n(\mathbb{C})}$$
where we have $u_g \in \mathbb{C}[G]$ acting via the left regular
representation on $l^{\infty}[G]$.  This commuting square is symmetric,
so iterating the basic construction gives a subfactor $N \subset M$.
The conditions of lemma 4.2 are satisfied by the
$u_g$'s, so this subfactor
is $N \subset N \rtimes G$.

Likewise we may consider
$$\commsq{\mathbb{C}}{\mathbb{C}[G]}{l^{\infty}(G)}{M_n(\mathbb{C})}$$
In this case we take $v_g$ to be the unitary
acting on $l^{\infty}[G]$ via the right regular representation.
$g \rightarrow Ad v_g$ is an action of $G$ on
$M_n(\mathbb{C})$.
Left and right mutiplication commute, and the fixed point algebra
of this action
is exactly $\mathbb{C}[G]$.  The conditions of lemma 4.1
are then satisfied, so the compatible outer action of $G$ on
$M$ induced by the $Ad v_g$'s gives $N=M^G$.

We will now construct a Bisch-Haagerup subfactor just
as in section 3.6.
For two groups $H$ and $K$, with $|H|=m$ and $K=|N|$, we
may use the above
commuting squares to
construct a composite diagram as follows:

$$\begin{array}{ccccc}
l^{\infty}(H) \otimes \mathbb{C}[K]&\subset&
M_m(\mathbb{C}) \otimes \mathbb{C}[K]&\subset&
M_m(\mathbb{C)} \otimes M_n(\mathbb{C})\\
\cup& &\cup& &\cup\\
l^{\infty}(H) \otimes 1&\subset&
M_m(\mathbb{C})\otimes 1&\subset&M_m(\mathbb{C}) \otimes l^{\infty}(K)\\
\cup& &\cup& &\cup\\
\mathbb{C}&\subset&\mathbb{C}[H] \otimes 1&\subset&
\mathbb{C}[H] \otimes l^{\infty}(K)\end{array}$$

The overall quadrilateral
$$\commsq{\mathbb{C}}{ \mathbb{C}[H] \otimes l^{\infty}(K)}
{l^{\infty}(H) \otimes \mathbb{C}[K]}{M_m(\mathbb{C}) \otimes
M_n(\mathbb{C})}$$ commutes, since the original two squares
commute and
$$E_{A \otimes B}(x \otimes y)=E_A(x) \otimes E_B(y)$$
Symmetry follows from symmetry of the original
squares, all inclusions are connected,
and all traces may be taken to be Markov
as in section 3.6.  Therefore this commuting square gives a subfactor
$N \subset M$ via iteration of the basic construction,
just as in the Hadamard case.

The top quadrilateral
$$\commsq
{l^{\infty}(H) \otimes 1}
{M_m(\mathbb{C}) \otimes l^{\infty}(K)}
{l^{\infty}(H) \otimes \mathbb{C}[K]}{M_m(\mathbb{C}) \otimes
M_n(\mathbb{C})}$$ commutes as well, so we again have
an intermediate subfactor $N \subset P \subset M$.
Let
$$B_0=l^{\infty}(H) \otimes 1$$
$$B_1=M_m(\mathbb{C}) \otimes l^{\infty}(K)d$$

$B_0 \subset B_1$ generates $P$ via iteration of
the basic construction.  As before
$$B_0' \cap B_1 = l^{\infty}(H) \otimes l^{\infty}(K)$$
and is abelian.

Let $a_h=v_h \otimes 1$, where
$v_h$ comes from the right regular represention of $H$.
These unitaries induce an action of
$H$ on $B_1$ which leaves $B_0$ invariant, such that
$B_0^H = \mathbb{C}[H] \otimes l^{\infty}(K)$.  So from
section 3.1,
$N=P^H$, where the action of $H$ is given by
the compatible automorphisms of $P$ induced
by $Ad a_h$.

Let $a_k = 1 \otimes u_k$.
Then the $a_k$'s along with $B_1$ generate
$M_m(\mathbb{C}) \otimes M_n(\mathbb{C})$, since
the original commuting square is nondegenerate,
and $B_0 \rtimes K = l^{\infty}(H) \otimes \mathbb{C}[K]$.
So from section 3.2, $M = P \rtimes K$.

The composite subfactor
$N \subset M$ is therefore of Bisch-Haagerup type.  Since $a_h$ and
$a_k$ commute, it is depth 2.

Let $T$ be a unitary in $l^{\infty}(H) \otimes l^{\infty}(K)$.
We then consider the inclusion of algebras

$$\begin{array}{ccc}
l^{\infty}(H) \otimes \mathbb{C}[K]&\subset&
M_m(\mathbb{C)} \otimes M_n(\mathbb{C})\\
\cup& &\cup\\
l^{\infty}(H) \otimes 1&\subset&
M_m(\mathbb{C}) \otimes l^{\infty}(K)\\
\cup& &\cup\\
\mathbb{C}&\subset&\
T(\mathbb{C}[H] \otimes l^{\infty}(K))T^*\end{array}$$

The top square still commutes.  The bottom quadrilateral
is obtained by applying $Ad T$ to the bottom
square of the original figure, since $T$ normalizes
$B_0$ and $B_1$.  Therefore the bottom square still
commutes, and is isomorphic to the bottom square in the
previous diagram.  It follows that the overall square commutes as well,
and we obtain $N \subset P \subset M$ by iteration
of the basic construction.
We again find that the subfactor is $P^H \subset P \rtimes K$,
where the action of $K$ is given by $Ad b_k = Ad (1 \otimes u_k)$,
but the action of $H$ is induced by $Ad b_h = Ad T(v_k \otimes 1)T^*$.
These actions are compatible with the tower of $P$, as in the Hadamard
case.

As before $H$ and $K$ act faithfully on $B_0' \cap B_1=
l^{\infty}(H) \otimes l^{\infty}(K)$
(although $H$ now acts by right multiplication).  Therefore
the first commutator subgroup $N$ of $H \ast K$
acts trivially on
$B_0' \cap B_1$.  As in section 3.6, since
this algebra is maximal abelian in $M_m(\mathbb{C}) \otimes M_n(\mathbb{C})$,
it follows that $N|_{B_1}$ is abelian; $N$ acts compatibly,
so its action on all of $P$ is abelian as well.
For $h \in H$, $k \in K$, we find
$hkh^{-1}k^{-1}$ is induced by
$Th(T^*)k(T^*)hk(T)$, exactly as in the Hadamard case.
Outerness of compatible automorphisms is as before, and
we may still compute principal graphs.  The computation
of the characteristic invariant is the same, but the result of~\cite{BNP}
will not apply,
so classifying up to subfactor isomorphism will be more difficult
than in the abelian case.

This construction allows us to find principal graphs for a larger class
of commuting-square subfactors, which includes all the Hadamard examples
mentioned previously.  We present one example here.

Let $H=\Z_2$, $K=S_3$, $T=(1,...,1,i)$.
Then there are five generators of $N$.  Each is order-2,
and their product is a nontrivial inner, so $N=\Z_2^4$.
We find that $|G|=192$.  The multiplication table of $G$
may be found just as in the Hadamard case, so the methods
of~\cite{BH} give the principal graph.  There
are 16 odd vertices.  Every element of $N$ is order 2,
so $h(hkhk^{-1})h=khk^-1h=hkhk^{-1}$, and $N$ commutes
with $H$; it follows that $N$ is equal to the commutant
of $H$ in $KN$.

This means that there are 16 double
cosets of size 2 in $G=HKN$, giving 32
vertices (each connected to one odd vertex),
and 40 of size 4, giving 40 more (each connected
to two odd vertices),
for a total of 72 even vertices.  Each odd vertex
connects to 7=2+5 even vertices.
The multiplication
table of $G$ is known, so the full principal 
graph may be constructed using~\cite{BH}, and the dual
principal graph in the same way.

\vspace{0.5cm}
\noindent{\it Acknowledgement.}
I am grateful to Professor Vaughan Jones for suggesting this problem to me,
and for describing the twisted tensor product construction.
The results of this paper were part of my doctoral thesis at
UC Berkeley \cite{RB}.

\end{document}